\documentclass[onefignum,onetabnum]{siamart171218}

\usepackage{lipsum}
\usepackage{amsfonts}
\usepackage{graphicx}
\usepackage{epstopdf}
\usepackage{algorithmic}
\usepackage{subcaption}
\usepackage{bbm}

\ifpdf
  \DeclareGraphicsExtensions{.eps,.pdf,.png,.jpg}
\else
  \DeclareGraphicsExtensions{.eps}
\fi

\usepackage{booktabs}
\usepackage{tabularx}

\newtheorem{remark}{Remark}

\newcommand{\prob}{\textnormal{Pr}}

\newcommand{\R}{\mathbb{R}}
\newcommand{\bbS}{\mathbb{S}}
\newcommand{\bbE}{\mathbb{E}}

\newcommand{\activbnd}{C}
\newcommand{\lipconst}{L}
\newcommand{\Hmin}{\lambda_{\min}(H)}

\def\hW{\widehat{W}}
\def\hH{\widehat{H}}
\def\hG{\widehat{G}}
\def\halpha{\widehat{\alpha}}

\def\cU{\mathcal{U}}
\def\cN{\mathcal{N}}

\usepackage{hyperref}
\definecolor{mylinkcolor}{RGB}{0,0,130}
\hypersetup{colorlinks,allcolors=mylinkcolor,citecolor=mylinkcolor}

\newcommand{\emthin}{\thinspace\textemdash\thinspace}

\crefname{section}{Section}{Sections}
\crefname{subsection}{Subsection}{Subsections}

\headers{Over-parametrized neural networks}{}
\title{Over-parametrized neural networks as under-determined linear systems}
\author{Austin R.\ Benson\thanks{Department of Computer Science, Cornell University, Ithaca, NY 14853-4201, United States (\email{arb@cs.cornell.edu}).}\and Anil Damle\thanks{Department of Computer Science, Cornell University, Ithaca, NY 14853-4201, United States (\email{damle@cornell.edu}).} \and Alex Townsend\thanks{Department of Mathematics, Cornell University, Ithaca, NY 14853-4201, United States (\email{townsend@cornell.edu}).}}

\usepackage{amsopn}

\ifpdf
\hypersetup{
  pdftitle={overparametrized},
  pdfauthor={}
}
\fi

\begin{document}

\maketitle

% REQUIRED
\begin{abstract}
We draw connections between simple neural networks and under-determined linear systems to comprehensively explore several interesting theoretical questions in the study of neural networks. First, we emphatically show that it is unsurprising such networks can achieve zero training loss. More specifically, we provide lower bounds on the width of a single hidden layer neural network such that only training the last linear layer suffices to reach zero training loss. Our lower bounds grow more slowly with data set size than existing work that trains the hidden layer weights. Second, we show that kernels typically associated with the ReLU activation function have fundamental flaws\emthin there are simple data sets where it is impossible for widely studied bias-free models to achieve zero training loss irrespective of how the parameters are chosen or trained. Lastly, our analysis of gradient descent clearly illustrates how spectral properties of certain matrices impact both the early iteration and long-term training behavior. We propose new activation functions that avoid the pitfalls of ReLU in that they admit zero training loss solutions for any set of distinct data points and experimentally exhibit favorable spectral properties. 
  \end{abstract}

% REQUIRED

% REQUIRED

% 1. It is ``easy'' to get zero training error without weight training (effectively no gap between the random features and neural tangent regimes, from a training point of view)
% 2. Networks with many common activation functions will struggle to get zero training error (including ReLU)
% 3. Our setup makes it easier to understand optimization

%%%%%%%%%%%%%%%%%%%%%%%%%%%%%%%%%%%%%%%%
\section{Introduction}
\label{sec:introduction}
%%%%%%%%%%%%%%%%%%%%%%%%%%%%%%%%%%%%%%%%
Neural networks are among the predominant mathematical models in machine learning, though theoreticians still struggle to understand their efficacy fully. While the origins of neural networks go back decades, they have recently benefited from modern computing architectures, the growth of labeled data sets, and their breadth of applicability. In this work, we use a simple mathematical model
to understand the properties of ``over-parameterized'' neural networks 
(roughly, those with more model parameters than are seemingly necessary)
through the lens of under-determined systems of equations in numerical analysis.

Motivated by empirical observations~\cite{zhang2017understanding}, there has been significant work on showing that neural networks can achieve zero training loss when the weights are trained using a simple optimization algorithm such as (stochastic) gradient descent~\cite{allen2018convergence,du2019gradient,du2018gradient,oymak2020toward}. However, such work depends delicately on analyzing networks in what is essentially a ``linearized regime''~\cite{chizat2019lazy,ghorbani2019limitations,lee2019wide}. In this regime, the training dynamics mostly follow a kernel method~\cite{jacot2018neural} called the \emph{neural tangent kernel} (NTK). Surprisingly, the property of achieving zero training error is effectively independent of training the weights of the network and the 
theoretical analysis relies on the fact that the hidden layer weights hardly change from their random initialization.

In this paper, we show that just training the last layer is sufficient to achieve zero training loss in most cases. The lower bounds we find on the network width to acquire this property are significantly smaller than those previously developed for training the hidden layer weights. More specifically, for a simple two-layer neural network, we provide conditions for the existence of zero training loss solutions and gradient descent's convergence to it when only the last layer weights are trained. Notably, our results hold for randomly selected weights and are distinct from results on finite sample expressiveness~\cite[Sect.~4]{zhang2017understanding}, where hidden layer weights are initialized in a data-dependent manner. When a network's weights are chosen at random, training the last layer is equivalent to solving an under-determined linear system, which is a well-understood problem in numerical analysis. The linear system corresponds to the so-called random feature regime~\cite{rahimi2008random,rahimi2008uniform,rahimi2009weighted}, except with feature choices motivated by the structure of a single layer fully connected neural network (see~\cref{sec:model}). The under-determined linear systems viewpoint turns out to be reasonably flexible and readily applicable to the use of ``pre-trained'' models.

To achieve our theoretical results, we use matrix concentration inequalities to connect the under-determined linear system to a square system in the infinite width limit. 
The study of the corresponding infinite width limit is used in classical analysis of neural networks via ridge function approximation theory~\cite{barron1993universal,cybenko1989approximation,leshno1993multilayer,pinkus1999approximation} and is closely related to kernel methods~\cite{neal1995bayesian,williams1997computing}. We can show that in the infinite width limit, one can achieve zero training loss by showing that the kernel induced by certain activation functions and random weights is strictly positive definite.  

In~\cref{sec:activation}, we rigorously characterize when an activation function corresponds to a positive definite kernel. A critical insight is that the commonly used rectified linear unit (ReLU) activation function is not strictly positive definite for a bias-free network. We provide a set of eight data points that demonstrates this. Furthermore, The same eight data points explicitly show that the NTK is also only positive semi-definite when using ReLU. Importantly, our results hold for any finite width and the infinite width limit. In the finite width setting, our results are independent of the choice of hidden layer weights and conclude that there are benign datasets for which a bias-free single hidden layer neural network with ReLU cannot generically achieve zero training loss. 

Motivated by these observations, in~\cref{sec:activation} we introduce a new set of activation functions from the radial basis function literature that yield positive definite kernels in the infinite limit and experimentally produce well-conditioned finite width models. Our analysis is closely related to recent connections between random features models and kernel methods~\cite{bach2017breaking,bach2017equivalence,daniely2016toward}, spectral properties of neural networks~\cite{bietti2019inductive,rahaman2019spectral,ronen2019convergence}, and generalization performance of kernel methods~\cite{belkin2018understand}. Our new activation functions complement these results by experimentally exhibiting favorable spectral properties. Furthermore, we make concrete connections between the spectral properties of finite systems and the training process in \cref{sec:training}.

More specifically, in \cref{sec:training}, we show that achieving zero training error via gradient descent applied to the last layer is relatively simple to accomplish. Thus, achieving zero training loss with over-parameterized networks and simple optimization methods is not surprising.
If gradient descent is initialized appropriately, the minimal norm solution is found. 
This fact alone is not particularly useful as it is easy to characterize the minimal norm solution of a consistent under-determined linear system. More interesting are the connections we develop between training the last layer and the so-called Landweber iteration~\cite{landweber1951iteration}. Our analysis shows the connection between spectral properties of the kernel and optimization performance, thereby motivating the choice of activation functions that correspond to positive definite kernels and have favorable spectral properties. The suggested activation functions have such properties, and simple numerical experiments show their efficacy both in early iteration performance and ability to achieve zero training error via gradient descent.

Our results show that there is effectively no gap between what is achievable by random features models and fully trained neural networks in terms of training loss. This does not mean the two models are equivalent more generally\emthin fully trained neural networks are more complex models and nominally more capable. In fact, recent research explores the theoretical distinctions between the random features and NTK regimes~\cite{ghorbani2019limitations,ghorbani2019linearized}. Further theoretical insight on the random features regime and the models properties can be drawn from its connections to ``ridgeless'' regression and interpolation~\cite{bartlett2020benign,belkin2018overfitting,hastie2019surprises}. However, many of these results are asymptotic (in data dimension, network width, and data set size) and make statistical assumptions on the data. In this work, we explicitly focus on finite problem sizes and deliberately make minimal assumptions on the input data. Lastly, while there is extensive work on the generalization properties of neural networks, even in the random features regime~\cite{mei2019generalization,rudi2017generalization}), we explicitly omit such a discussion here.

% Nevertheless, early work explored the benefits of ridge functions for high-dimensional problems~\cite{donoho1989projection} and explored ``discrete'' and stable representations by ridge functions~\cite{candes1999harmonic}. 

%%%%%%%%%%%%%%%%%%%%%%%%%%%%%%%%%%%%%%%%
\section{An over-parameterized two-layer neural network}\label{sec:model}

%\subsection{Training data and model}
Suppose that we are given $n$ distinct data points $x_1,\ldots,x_n\in \R^d$ that are normalized so that $\|x_i\|_2=1$ for $1\leq i\leq n$,\footnote{The normalization assumption simplifies the exposition and relates the model to an approximation theory problem on the sphere. Some of our results may be extended to general non-zero data, at the expense of additional constants depending on the relative norms of data points and details of the specific activation functions.} where each data point is assigned a label $y_1,\ldots, y_n\in\R$. We let $X\in\R^{d\times n}$ be the data matrix, where $X(:,i) = x_i$ and $y = \begin{pmatrix} y_1 & \cdots & y_n\end{pmatrix}$ is the vector of labels.

We consider the task of training a two-layer neural network, $G\colon\mathbb{R}^d\rightarrow \mathbb{R}$, with a single fully-connected hidden layer of width $m$ and linear last layer such that $G(x_i) \approx y_i$ for $1\leq i\leq n$. More specifically, for a fixed continuous scalar-valued activation function $\gamma \colon \R\rightarrow \R$ and integer $m$, $G$ takes the form\footnote{We omit bias terms in the model. They can be implicitly added by appending constants to the data points after normalization.} 
\[
G(x) = m^{-1/2}\gamma(x^TW)\alpha,
\] 
where $W\in\R^{d\times m}$ is the weight matrix, $\alpha\in\R^m$ is the last layer, and $\gamma(x^TW)$ is the $1\times m$ vector obtained by applying $\gamma$ to $x^TW\in\R^{1\times m}$ entrywise. Throughout this paper, we assume the activation function $\gamma$ is Lipschitz continuous with Lipschitz constant $\lipconst$, and its magnitude is bounded by $\activbnd$ on $[-1,1]$, i.e., $\lvert \gamma(z)\rvert \leq \activbnd$ for $z\in [-1,1].$ For many common activation functions it suffices to take $\activbnd = 1.$

One usually fits $G$ by jointly learning $W$ and $\alpha$ via the least-squares problem
\begin{equation}
\label{eqn:onelayer_model}
\min_{W\in\R^{d\times m},\,\alpha\in\R^m}   \sum_{i=1}^n \left( \frac{1}{\sqrt{m}}\gamma (x_i^TW)\alpha - y_i\right)^2
\end{equation}
or, equivalently, 
\begin{equation}
\label{eqn:onelayer_modelM}
\min_{W\in\R^{d\times m},\,\alpha\in\R^m} \left\|\frac{1}{\sqrt{m}}\gamma(X^TW)\alpha - y\right\|_2^2.
\end{equation}
It is standard to try and solve~\cref{eqn:onelayer_model} by gradient descent or stochastic gradient descent. There are many popular choices for the activation function $\gamma$, and we will analyze specific choices later. The choice of $\gamma$ affects how computationally expensive it is to solve~\cref{eqn:onelayer_model} and the quality of the final discovered $G$.

When $m> n$, we consider the model to be over parametrized since there are more hidden nodes than training data points. When there are more parameters than data points in the neural network, it is quite plausible to find an interpolating $G$ so that $G(x_i)=y_i$ for $1\leq i\leq n$. In other words, we may expect the existence of $\widetilde{W}$ and $\widetilde{\alpha}$ such that
\begin{equation*}
\frac{1}{\sqrt{m}}\gamma(X^T\widetilde{W})\widetilde{\alpha} = y.
\end{equation*}
Surprisingly, we find that some of the most popular choices of $\gamma$ do not allow one to find an interpolating $G$; regardless of how over-parameterized the network is. Concretely, we show this for the commonly used ReLU activation function, which is defined as $\gamma(z)=(z)_+,$ where $(z)_+ = z$ if $z>0$ and $0$ for $z\leq 0$. \Cref{thm:ReLUBad} shows that for this model ReLU does not admit interpolating solutions when using the well-separated data points in \cref{def:ReLUHatesThis}. 

\begin{definition}
\label{def:ReLUHatesThis}
Let $d\geq 3$. The dataset $X_R = \{x_1,\ldots,x_8\}$ is given by  
\begin{equation} 
\begin{aligned}
x_1 &= \left(s,s,s,\mathbf{0}\right), &x_2 = (-s,s,s,\mathbf{0}),\,\,&x_3 = (s,-s,s,\mathbf{0}), &x_4 = (s,s,-s,\mathbf{0}),\\
x_5 &= \left(-s,-s,s,\mathbf{0}\right), &x_6 = (-s,s,-s,\mathbf{0}), \,\,&x_7 = (s,-s,-s,\mathbf{0}),  &x_8 = (-s,-s,-s,\mathbf{0}),\\
%X = \left( \pm \frac{1}{\sqrt{3}}, \pm \frac{1}{\sqrt{3}}, \pm \frac{1}{\sqrt{3}} \right).
\end{aligned} 
\label{eq:ReLUHatesThis}
\end{equation} 
where $\mathbf{0}\in\R^{d-3}$ is the zero vector and $s = 1/\sqrt{3}$. 
\end{definition} 

\begin{theorem} 
Let $d\geq 3$ and $x_1,\ldots,x_8\in\mathbb{R}^d$ be those in~\cref{def:ReLUHatesThis}.  The matrix $E\in\mathbb{R}^{8\times m}$ given by $E_{ij} = (x_i^Tw_j)_+$ is of rank at most $7$ for any $w_1,\ldots,w_m\in\mathbb{R}^d$.
%\[
%\hH_{ij} = \frac{1}{m}\sum_{k=1}^m  (x_i^Tw_k)_+ (w_k^Tx_j)_+ = \frac{1}{m} E^TE, \qquad E_{kj} = (w_k^Tx_j)_+ 
%\]
%is singular.  and $w_k\in\mathbb{S}^{d-1}$ for $1\leq k\leq m$. 
\label{thm:ReLUBad} 
\end{theorem} 
\begin{proof} 
%We show that $E\in\mathbb{C}^{m\times 8}$ is of rank at most $7$ implying that the $8\times 8$ matrix $\hH$ is of rank at most $7$ and hence, singular.  
Let $v = \left(1,-1,-1,-1, 1, 1,1,-1\right)^T$. For any $1\leq j\leq m$, we have 
\[
\begin{aligned}
\sqrt{3}(E^Tv)_j = \sqrt{3}\sum_{i=1}^8 v_i(w_j^Tx_i)_+  & = (w_{j,1}+w_{j,2}+w_{j,3})_+ - (-w_{j,1}-w_{j,2}-w_{j,3})_+ \\[-10pt]
&  + (w_{j,1}-w_{j,2}-w_{j,3})_+ -  (-w_{j,1}+w_{j,2}+w_{j,3})_+ \\
&  + (-w_{j,1}+w_{j,2}-w_{j,3})_+- (w_{j,1}-w_{j,2}+w_{j,3})_+ \\
& + (-w_{j,1}-w_{j,2}+w_{j,3})_+ - (w_{j,1}+w_{j,2}-w_{j,3})_+\\
& = (w_{j,1}+w_{j,2}+w_{j,3}) + (w_{j,1}-w_{j,2}-w_{j,3}) \\
& + (-w_{j,1}+w_{j,2}-w_{j,3}) + (-w_{j,1}-w_{j,2}+w_{j,3}) \\
& = 0. 
\end{aligned} 
\]
Therefore, $E^Tv = 0$ and $E$ is of rank at most $7$. 
\end{proof}

The implications of~\cref{thm:ReLUBad} are quite significant, particularly when we generalize the result to properties of the induced kernels in \cref{sec:activation}, and are summarized in~\cref{cor:noInterp}. In particular, there are ``nice'' data sets $X$ containing distinct and well-separated points for which~\cref{eqn:onelayer_model} may not exactly fit all possible labels $y$ when the ReLU activation function is used. More precisely, it shows that for ReLU, $\gamma(X_R^TW)$ is not full row-rank regardless of the width $m$ and choice of weights $W.$ Consequently, there are sets of labels $y$ outside the range of $\gamma(X_R^TW),$ which precludes the possibility of generically achieving zero training loss regardless of how $W$ and $\alpha$ are trained. Practically, \cref{rem:bias} shows how more commonly used models may avoid this theoretical pitfall. Nevertheless, the existence of such data sets may help shed light on what makes problems difficult for practical single hidden layer neural networks.

Importantly, planting~\cref{eq:ReLUHatesThis} into any other data set makes solving~\cref{eqn:onelayer_model} challenging as $\gamma(X^TW)$ remains rank-deficient. An additional implication of having these ``singular'' point sets is that for any $n \geq 8$, we can make the minimal singular value of $\gamma(X^TW)$ arbitrarily small while enforcing that $X$ has no parallel points. This observation has consequences on the width of networks shown to achieve zero training loss in prior work~\cite{du2019gradient,du2018gradient}. 

\begin{corollary}
\label{cor:noInterp}
Let $d\geq 3$ and $(X, y)$ be any dataset that contains the points in \cref{def:ReLUHatesThis}. Then, there exists a non-zero vector $q\in\R^n$ such that if $y^Tq\neq 0$ no bias-free single hidden layer neural network using ReLU activation, i.e., $\gamma(z) = (z)_+$ in \cref{eqn:onelayer_model} can achieve zero error for~\cref{eqn:onelayer_modelM}.
\end{corollary}
\begin{proof}
Without a loss of generality, assume that $X$ is ordered such that $X = \begin{bmatrix} X_R & X_C \end{bmatrix}$ and $y = \begin{pmatrix}y_R & y_C\end{pmatrix},$ where $y_R\in\R^8$ are the labels associated with $X_R$ from \cref{def:ReLUHatesThis}. Set $v = \left(1,-1,-1,-1, 1, 1,1,-1\right)^T$ as in the proof of \cref{thm:ReLUBad}. We will show that the fixed vector $q = \begin{pmatrix} v & 0 \end{pmatrix}^T$ satisfies the criteria of the Corollary; assume that $y^Tq\neq 0,$ which implies that $y_R^Tv\neq 0.$ 

Now, suppose that there are weights $W$, and $\alpha$ that achieved zero error for~\cref{eqn:onelayer_modelM}. This implies that $m^{-1/2}\gamma(X_R^TW)\alpha = y_R$ and, therefore, $m^{-1/2}v^T\gamma(X_R^TW)\alpha = v^Ty_R.$ By~\cref{thm:ReLUBad}, we find that $v^T\gamma(X_R^TW) = 0,$ so $y_R^Tv\neq 0$ leads to a contradiction. 
\end{proof}

\begin{remark}
\label{rem:bias}
While we formally allow for the inclusion of bias terms through augmentation of the data, this technique formally constrains the set of allowable $X.$ These constrains may suffice to ensure that $(X^T\widetilde{W})_+$ is full row-rank for some $W$ provided the columns of $X$ are distinct. Consequently, the bias terms are theoretically quite important for a ReLU network to be generically able to achieve zero training loss\emthin this agrees with the universal approximation theorem~\cite{leshno1993multilayer} and finite width expressiveness results~\cite[Sect.~4]{zhang2017understanding}.
\end{remark}

\subsection{The random features regime}
\label{sec:population}
In light of~\cref{eqn:onelayer_modelM}, when is it reasonable to expect an interpolating solution to exist? And, if it does, then can gradient descent find it? To answer this question, we study~\cref{eqn:onelayer_model} in the so-called {\em random features regime}~\cite{rahimi2008random}, where $W$ is arbitrary and kept fixed while one optimizes over $\alpha$. In this regime, provided that $m\geq n$,~\cref{eqn:onelayer_model} reduces to an under-determined least-squares problem\footnote{If $m<n$, then~\cref{eqn:onelayer_modelM} cannot generically have a solution that achieves zero training loss.} given by  
\begin{equation}
\label{eqn:LS_alpha}
 \min_{\alpha\in \R^m} \left\| \frac{1}{\sqrt{m}} \gamma(X^TW)\alpha - y\right\|_2^2.
\end{equation}
Characterizing solutions to~\cref{eqn:LS_alpha} requires understanding the properties of the matrix $\gamma(X^TW)$ for a fixed $W$. For example, if $\gamma(X^TW)$ has full row-rank, then the minimal norm solution to~\cref{eqn:LS_alpha} is given by $\halpha = \frac{1}{\sqrt{m}}\gamma(W^TX)z$, where $z$ solves the non-singular linear system
\begin{equation} 
\frac{1}{m}\gamma(X^TW)\gamma(W^TX)z = y.
\label{eq:NormalEquations}
\end{equation} 
This means that it is relatively easy to obtain zero training loss provided that $\gamma(X^TW)$ has full row-rank or, equivalently, the matrix in~\cref{eq:NormalEquations} is non-singular.  

Mirroring standard initialization techniques, we study~\cref{eqn:LS_alpha} when $W$ is a random matrix with each column is drawn independently from the uniform distribution over the sphere, i.e., $w_i\sim \cU(\bbS^{d-1})$. Denoting the matrix in~\cref{eq:NormalEquations} as $\hH,$ it is given by
\begin{equation}
\hH_{ij} \equiv \frac{1}{m}\gamma(x_i^TW)\gamma(W^Tx_j) = \frac{1}{m}\sum_{k=1}^{m}\gamma(x_i^Tw_k)\gamma(w_k^Tx_j).
\label{eq:Hhat}
\end{equation}
We would like to ensure that $\hH$ is invertible with high probability when $m$ is sufficiently large and accomplish this by connecting $\hH$ to a matrix that characterizes the infinite width limit of our simple neural network.

\subsubsection{Infinite width limit}\label{sec:InfiniteWidth}
For a given $\gamma$, the entries of the population-level matrix $H\in\R^{n\times n}$ are given by 
\begin{equation}
\label{eqn:H}
%\begin{aligned}
H_{ij} = \bbE_{w\sim \cU(\bbS^{d-1})} \! \left[\gamma(x_i^Tw)\gamma(w^Tx_j)\right]
= \frac{1}{\lvert \bbS^{d-1}\rvert}\int_{\bbS^{d-1}}\gamma(x_i^Tw)\gamma(w^Tx_j)dw. 
%\end{aligned}
\end{equation}
Since $H$ is a Gram matrix, we know that $H$ is symmetric and $H\succeq 0$ for all continuous $\gamma$. Notably, the matrix $H$ appears when analyzing the infinite width limit of neural networks and their relation to Gaussian processes~\cite{neal1995bayesian,williams1997computing}. Throughout this paper, we often write $\bbE_{w}$ instead of $\bbE_{w\sim \cU(\bbS^{d-1})}$ to keep the notation compact. As we will show, if $H\succ 0$ it follows that $\gamma(X^TW)$ is full row-rank for sufficiently large $m.$ This implies that~\cref{eqn:LS_alpha} has a unique minimum norm solution with a zero objective value. The specific choice of $\gamma$ can impact the properties of $H$ and how they relate to training (see~\cref{sec:training}). In~\cref{sec:activation}, we chose a specific $\gamma$ motivated by bounding $\Hmin$ away from zero.

\begin{remark}
While it is tempting to try to find a connection between the matrix $H$ in~\cref{eqn:H} and a kernel interpolation matrix, i.e., $A_{jk} = \phi(x_j,x_k)$, where $\phi\colon \mathbb{S}^{d-1}\times \mathbb{S}^{d-1} \rightarrow \R$ is a positive definite reproducing kernel, it is typically not possible. To see this, note that the uniform distribution prescribes the inner product used to construct entries of $H$, and it is often not in agreement with the inner product needed to produce a proper kernel matrix~\cite{bach2017equivalence,hofmann2008kernel}. The more appropriate analogy is to that of Random Fourier Features~\cite{rahimi2008random} and Random Kitchen Sinks~\cite{rahimi2009weighted}.
\end{remark}

\subsection{Finite width}\label{sec:finite}
For random independent and identically distributed (i.i.d.) $w_i$, we have that $\hH$ is a (consistent) stochastic estimator of $H_{ij}$ as $m\rightarrow\infty$. Two natural questions are: 
(1) ``Is the minimal eigenvalue of $\hH$ bounded away from zero, and, if so, how large is it?'', and
(2) ``How well does $\hH$ approximate $H$?''.
Both answers depend on the value of $m$ and, implicitly through spectral properties of $H,$ the choice of $\gamma.$ 

First, we control $\lambda_{\min}(\hH)$ with high probability by using a matrix Chernoff bound. 
\begin{lemma}
\label{lem:evalue}
Let $0<\delta <1$. If $m \geq 10n\activbnd^2 \log(2n/\delta) /\Hmin$, then
\[
\lambda_{\min}(\hH) > \frac{\Hmin}{2}\quad \text{and} \quad\lambda_{\max}(\hH) < 2\lambda_{\max}(H)
\]
with probability at least $1-\delta.$
\end{lemma}
\begin{proof}
Since $w_i \in \bbS^{d-1}$ are independently drawn from $\cU(\bbS^{d-1}),$ the matrices $\frac{1}{m}\gamma(X^Tw_i)\gamma(w_i^TX)$ are independent and identically distributed positive semi-definite matrices. Furthermore, they satisfy $\bbE_w \!\!\left[\gamma(X^Tw)\gamma(w^TX)\right] = H$, and we have that
\[
\frac{1}{m}\|\gamma(X^Tw)\gamma(w^TX)\|_2\leq \frac{n\activbnd^2}{m}, \qquad w \in \bbS^{d-1},
\]
where $\activbnd$ is the upper bound on $\lvert \gamma\rvert$ from \cref{sec:model}. Based on these observations, the statement follows from a matrix Chernoff bound~\cite[Thm.~5.1.1]{tropp2015intro} combined with a union bound to simultaneously control the probability bounds on both the minimal and maximal eigenvalue.
\end{proof}
Using a matrix Bernstein inequality, we can understand how well $\hH$ approximates $H$ with high probability in the spectral norm.
\begin{lemma}
\label{lem:Happrox}
Let $0 < \delta < 1$. If 
\begin{equation} 
m \geq \frac{2n\activbnd^2}{\Hmin}\!\left(\kappa(H) + \frac{2}{3}\right)\log(2n/\delta),
\label{eq:mbound}
 \end{equation} 
then $\Pr\{\|\hH-H\|_2 < \Hmin/2\} > 1-\delta$. Here, $\kappa(H)$ denotes the condition number of $H$. 
\end{lemma}
\begin{proof}
First, note that 
\[
\bbE_w \!\left[\gamma(X^Tw)\gamma(w^TX)\gamma(X^Tw)\gamma(w^TX)\right] \preceq n\activbnd^2 \bbE_w \!\left[\gamma(X^Tw)\gamma(w^TX)\right] = n\activbnd^2H,
\]
which implies that 
\[
\left\| \bbE_w\! \left[\gamma(X^Tw)\gamma(w^TX)\gamma(X^Tw)\gamma(w^TX)\right] \right\|_2 \leq n\activbnd^2\|H\|_2.
\]

Since the weights $w_i \in \bbS^{d-1}$ are independently drawn from $\cU(\bbS^{d-1})$, we may invoke a matrix Bernstein inequality (see~\cite[Cor.~6.2.1]{tropp2015intro} and~\cite[Sec.~6.5]{tropp2015intro}) to obtain
\[
\Pr\{\|\hH-H\|_2 \geq \epsilon\} \leq 2n\exp\!\left(\frac{-m\epsilon^2/2}{n\activbnd^2\|H\|_2+2n\activbnd^2\epsilon/3}\right)
\]
for any $\epsilon > 0$. The result follows when $m$ satisfies~\cref{eq:mbound} by setting $\epsilon = \Hmin/2$ and using the fact that $\|H\|_2/\Hmin = \kappa(H)$. 
\end{proof}

It is useful to inspect the lower bounds on $m$ in~\cref{lem:evalue,lem:Happrox}. First, since $\Hmin \leq e_i^THe_i \leq \activbnd^2,$ the lower bounds on $m$ grow like $(n\log n)/\Hmin$. For fixed data points, this means that optimizing the lower bounds on $m$ is purely a question of choosing $\gamma$ such that $\Hmin$ is not too small, and $H$ is well-conditioned. Moreover, the appearance of $C^2/\Hmin$ is natural as the definition of $H$ includes two copies of $\gamma.$ If we scale $\gamma$ by some constant $p$, so that $C$ becomes $pC$, then we also scale $\Hmin$ by $p^2.$ Therefore, our bounds are properly invariant with respect to scalings of $\gamma.$

\begin{remark}
While we have restricted our discussion here to $w\sim\cU(\bbS^{d-1})$, analogous results are possible to derive for $w\sim \cN(0,\sigma^2I)$ by using sub-Gaussian versions of the matrix Bernstein inequality (see~\cref{sec:NormalBounds}). 
\end{remark}

%%%%%%%%%%%%%%%%%%%%%%%%%%%%%%%%%%%%%%%%
\section{Activation functions and their properties}
\label{sec:activation}
%%%%%%%%%%%%%%%%%%%%%%%%%%%%%%%%%%%%%%%%
\Cref{lem:evalue,lem:Happrox} show that for sufficiently large $m$ the stochastic estimator $\hH$ is a good approximation to the population matrix $H$, and many of the properties of $H$ are reflected in the simple model~\cref{eqn:onelayer_modelM}. This makes the choice of $\gamma$ crucial in both the random features regime and the lazy training regime~\cite{chizat2019lazy}. First, we show that the popular ReLU activation function can exhibit catastrophically bad behavior even for relatively benign data sets and we extend that analysis to the NTK. This result also extends to the swish activation function (see~\cref{sec:swish}). Second, we characterize the activation functions that led to a strictly positive definite $H$ for all data sets.

\subsection{Failure of the ReLU activation function}\label{sec:FailureReLU}
In~\cref{sec:model}, our simple model shows that the ReLU activation function has some severe shortcomings when used in~\cref{eqn:onelayer_model} (see~\cref{thm:ReLUBad}). We now show that the $8\times 8$ matrix $H$ in~\cref{eqn:H} is also singular for the data points given in~\cref{def:ReLUHatesThis}.\footnote{As shown in \cref{thm:ReLUBad} this data set also makes $\hH$ singular for \emph{any} choice of $m$ and $W.$}
\begin{theorem}
\label{thm:ReLUpoison}
Let $d\geq 3$. The matrix
\[
H_{ij} = \frac{1}{\left|\mathbb{S}^{d-1}\right|}\int_{\mathbb{S}^{d-1}} \left(x_i^Tw\right)_+\left(w^Tx_j\right)_+dw, \qquad 1\leq i,j\leq 8
\]
is singular, where $x_1,\ldots,x_8$ are given in~\cref{eq:ReLUHatesThis}. %Moreover, we find 
%Let $d\geq 3$, $n\geq 8$, and consider the points, 
%
%for all possible sign choices. For any data matrix $X \in\R^{d\times n}$ such that $X(:,\cJ) = X_R$ for some $\cJ\subseteq [n]$ with $\lvert \cJ \rvert = 8$ we have that
%\[
%\lambda_{\min}(H(X,\gamma_R,\cU(\bbS^{d-1})) = 0, \qquad \lambda_{\min}(\hH(X,\gamma_R,W)) = 0
%\]
%for any $m$ and $W\in\R^{d\times m}.$
\end{theorem}
\begin{proof}
First, we note that $H_{ij} = \phi(x_i^Tx_j)$, where~\cite{cho2010large,cho2009kernel}\footnote{Note that our expression in~\cref{eqn:ReLUfeature} differs slightly from that in~\cite{cho2010large}. The reference appears to contain a minor typographical error omitting a constant multiplicative factor. This omission has no bearing on the results herein, and our calculation agrees with~\cite{bach2017breaking}.} 
\begin{equation}
\begin{aligned}
\label{eqn:ReLUfeature}
%\phi(x^Ty) &= \frac{1}{\lvert \bbS^{d-1}\rvert}\int_{\mathbb{S}^{d-1}} (x^Tw)_+ (y^Tw)_+dw,\\
\phi(t) &=  \frac{\sin (\cos^{-1}(t)) + (\pi/2 - \cos^{-1}(t) )t}{2d\pi} + \frac{t}{4d}.
\end{aligned}
\end{equation}
Now, let $x_1,\ldots,x_8$ be the point set given in~\cref{eq:ReLUHatesThis}. Let $X\in\mathbb{R}^{8\times d}$ be the matrix such that the $k$th row of $X$ is $x_k$ for $1\leq k\leq 8$. The matrix of inner-products is given by 
\begin{equation} 
XX^T = \begin{bmatrix} 
1 & s^2 & s^2  & s^2  & -s^2  & -s^2  & -s^2  & -1\\ 
s^2  & 1 & -s^2  & -s^2  & s^2  & s^2  & -1 & -s^2\\ 
s^2 & -s^2 & 1 & -s^2 & s^2 & -1 & s^2 & -s^2\\ 
s^2 & -s^2 & -s^2 & 1 & -1 & s^2 & s^2 & -s^2\\ 
-s^2 & s^2 & s^2 & -1 & 1 & -s^2 & -s^2 & s^2\\ 
-s^2 & s^2 & -1 & s^2 & -s^2 & 1 & -s^2 & s^2\\
-s^2 & -1 & s^2 & s^2 & -s^2 & -s^2 & 1 & s^2\\ 
-1 &  -s^2 & -s^2 & -s^2 & s^2 & s^2 & s^2 & 1 
 \end{bmatrix},
\label{eq:InnerProducts} 
\end{equation} 
where $s = 1/\sqrt{3}$. We conclude that the $8\times 8$ matrix $H = \phi(XX^T)$ is singular by verifying that the vector $v = \left(1,-1,-1,-1, 1, 1,1,-1\right)^T$ is an eigenvector corresponding to an eigenvalue of $0$. Note that
\[
Hv = \phi(XX^T)v = \left(\phi(1) - 3\phi(s^2) + 3\phi(-s^2) - \phi(-1)\right)v.
\]
We find that $\phi(1) - 3\phi(s^2) + 3\phi(-s^2) - \phi(-1) = 0$ since
\[
\begin{aligned} 
\phi(1) - 3\phi(s^2) + 3\phi(-s^2) - \phi(-1) & = \frac{\pi - 2\sqrt{2} - \pi + \cos^{-1}(\tfrac{1}{3})  + 2\sqrt{2} -\pi +  \cos^{-1}(\tfrac{1}{3})}{2d\pi} \\
&= \frac{-\pi + \cos^{-1}(\tfrac{1}{3}) + \cos^{-1}(-\tfrac{1}{3})}{2d\pi} = 0. 
\end{aligned} 
\]
This means that $Hv = 0$ and the matrix $H$ is singular. 
\end{proof}

\subsection{Failure of the ReLU Neural Tangent Kernel}\label{sec:failureNTK}
In situations where the weights are randomly drawn and then trained, the NTK has gained popularity~\cite{jacot2018neural} and it is particularly relevant in regimes where training leads to small changes in the weights~\cite{chizat2019lazy,du2018gradient,lee2019wide}. We find that the NTK for ReLU has the same severe shortcomings as the activation function itself. When the weights are initialized on $\mathbb{S}^{d-1}$, the matrix of interest for the NTK is
\begin{equation}
\label{eqn:NTKmatrix}
G_{ij} = \bbE_{w\sim \cU(\bbS^{d-1})}\!\left[x_i^Tx_j\mathbbm{1}_{x_i^Tw>0}\mathbbm{1}_{w^Tx_j>0}\right] = \frac{1}{\lvert \bbS^{d-1}\rvert}\int_{\bbS^{d-1}}\!\!\!x_i^Tx_j\mathbbm{1}_{x_i^Tw>0}\mathbbm{1}_{w^Tx_j>0}dw,
\end{equation}
where $\mathbbm{1}_{z>0}$ denotes the indicator function for $z>0$ (\emph{i.e.,} $\mathbbm{1}_{z>0}=1$ if $z>0$ and 0 otherwise). Our key observation is that the point set in~\cref{def:ReLUHatesThis} also causes $G$ to be singular.
\begin{theorem}
\label{thm:NTKpoison}
Let $d\geq 3$ and consider the point set in~\cref{def:ReLUHatesThis}. With this point set, the $8\times 8$ matrix $G$ in~\cref{eqn:NTKmatrix}, and the matrix
\[
\hG_{jk} = \frac{1}{m}\sum_{i=1}^m x_j^Tx_k\mathbbm{1}_{x_j^Tw_i>0}\mathbbm{1}_{x_k^Tw_i>0}
\]
are singular. Here, $x_1,\ldots,x_8$ are given in~\cref{def:ReLUHatesThis} and $w_k\in\mathbb{S}^{d-1}$ for $1\leq k\leq m$. 
%\begin{equation} 
%x_1 = \frac{1}{\sqrt{3}}\begin{bmatrix}1\\1\\1 \end{bmatrix}, \quad x_2 = \frac{1}{\sqrt{3}}\begin{bmatrix}1\\1\\-1 \end{bmatrix}, \quad x_3 = \frac{1}{\sqrt{3}}\begin{bmatrix}-1\\-1\\1 \end{bmatrix},\quad x_4 = \frac{1}{\sqrt{3}}\begin{bmatrix}-1\\-1\\-1 \end{bmatrix}.
%\label{eq:BadSet}
%\end{equation} 
%With the point set in~\cref{eq:BadSet} the $4\times 4$ matrix in~\cref{eqn:NTKmatrix} is singular. 
\end{theorem}
\begin{proof}
We first show that $G$ is singular. We note that $G_{ij} = \tilde{\phi}(x_i^Tx_j)$ with~\cite{arora2019fine,lee2019wide}\footnote{Our expression for $\tilde{\phi}$ differs by a constant factor~\cite{cho2010large}. Again, this has no bearing on the results herein.}
\[
\tilde{\phi}(x^Ty)= (x^Ty)\frac{\pi-\cos^{-1}(x^Ty)}{2d\pi}.
\]
In a similar way to the proof of~\cref{thm:ReLUpoison}, one can verify that $Gv = \phi(XX^T)v = 0$ with $v = \left(1,-1,-1,-1, 1, 1,1,-1\right)^T$. For the matrix $\hG$, we find that $\hG_{ij} = (x_i^T\tilde{E}^T\tilde{E}x_j)/m$ with $\tilde{E}_{kj} = \mathbbm{1}_{w_k^Tx_j>0}$, and one can verify that $\tilde{E}\tilde{v} = 0$ with $\tilde{v} = \left(1,0,0,-1,-1,0,0,1\right)^T$. 
%Since $\tilde{\phi}(\theta) - \theta /(4(d+1))$ is an even function in $\theta$, $\tilde{K}(x,y)$ is not a strictly positive definite kernel (see~\cref{lem:characterization}). To make this more concrete, consider the point set of size $8$ in~\cref{eq:ReLUHatesThis} with the matrix of inner-products given in~\cref{eq:InnerProducts}. Finally, verify that the vector $v = \left(1,-1,-1,-1, 1, 1,1,-1\right)^T$ is in the null space of the matrix $A$ where $A_{jk} = \tilde{K}(x_j,x_k)$. This follows since
%\[
%\tilde{\phi}(1) - 3\tilde{\phi}(1/3) + 3\tilde{\phi}(-1/3) - \tilde{\phi}(-1) = 0. 
%\] 
\end{proof} 

Since the point set in~\cref{def:ReLUHatesThis} causes $H$, $\hH$, $G$, and $\hG$ to be singular and~\cref{thm:ReLUpoison} holds for any $W$, one cannot overcome this catastrophic behavior by jointly training $W$ and $\alpha$ in~\cref{eqn:onelayer_model}. Moreover, this shows that without placing additional assumptions on the data or modifying the simple model it is impossible to generically assume $H$ and $G$ are non-singular.\footnote{Often an assumption is made that no two data points are parallel~\cite{du2018gradient}. However, for a classification problem on the sphere this is rather unnatural as distinct parallel points are maximally separated with respect to their inner-product and their geodesic distance on the sphere.} Moreover, prior results~\cite{menegatto1992interpolation,menegatto1994strictly} adapted to this setting suggests there are many point sets that share the properties of those in~\cref{def:ReLUHatesThis} when the kernel function does not have both an even and odd part that is non-polynomial. Lastly, For $H$ and $G$ these results are related to observations in~\cite{ronen2019convergence}, where the relative effectiveness of learning distinct components of an underlying function is studied. 

\begin{remark}
While we have followed the closely related literature (particularly that on achieving zero training loss) by omitting bias terms explicitly, as alluded to in \cref{rem:bias} the addition of bias terms effectively constrain $x_i$ to certain parts of the sphere. It is possible the minimal eigenvalues of $H$ and $G$ are bounded away from zero when these additional restrictions are placed on $x_i,$ making the inclusion of bias terms vital theoretically. Notably, the specific kernels we analyze are extensively studied (see, \emph{e.g.,} \cite{arora2019exact} and~\cite[Appendix C]{lee2019wide}) and therefore these results have important consequences for studying neural networks via approximation properties of the kernel methods induced by their infinite width limit.
\end{remark} 

\subsection{General activation functions and relations to kernels}
While~\cref{sec:FailureReLU,sec:failureNTK} demonstrate fundamental problems with ReLU, we would like to find sufficient conditions on $\gamma$ to ensure that $H$ in~\cref{eqn:H} is strictly positive definite for all data sets containing distinct points. Since the distribution in the expectation in~\cref{eqn:H} is $\cU(\bbS^{d-1})$, which is a rotationally invariant distribution, we know that for any continuous $\gamma$ there exists a function $\phi : [-1,1]\rightarrow \R$ such that
\begin{equation}
\label{eqn:H_phi}
H_{ij} = \phi(x_i^Tx_j).
\end{equation}
We can use the reproducing kernel literature on $\mathbb{S}^{d-1}$ to give a sufficient condition~\cite{smola2001regularization,hofmann2008kernel}.
%\Cref{lem:PosDefphi} provides an immediate characterization of when $H \succ 0$ for all data sets $X$ containing distinct points.
\begin{theorem}
For $d\geq 3$, the matrix $H$ is strictly positive definite for any distinct data points $x_1,\ldots,x_n\in\mathbb{S}^{d-1}$ if $\gamma$ is Lipschitz continuous on $[-1,1]$ with
\begin{equation} 
\gamma(t) = \sum_{k=0}^\infty a_k C_k^{((d-2)/2)}(t),  \qquad Z_\gamma = \{k\in\mathbb{Z}_+ : a_k\neq 0\},
\label{eq:ultraSexpansion}
\end{equation} 
such that $Z_\gamma$ contains infinitely many odd integers as well as infinitely many even integers. Here, $C_k^{(\tau)}$ is the ultraspherical polynomial of degree $k$ with parameter $\tau>0$. 
%For $d \geq 3$, the kernel $K(x,y) = \phi(x^Ty)$ is strictly positive definite on $\mathbb{S}^{d-1}\times \mathbb{S}^{d-1}$ if and only if 
%\[
%\phi(t) = \sum_{k=0}^\infty a_k C_k^{((d-2)/2)}(\cos t),  \qquad Z_\phi = \{k\in\mathbb{Z}_+ : a_k>0\},
%\]
%where $Z_\phi$ contains infinitely many odd integers as well as infinitely many even ones. Here, $C_k^{(\tau)}$ is the ultraspherical polynomial of degree $k$ with parameter $\tau>0$. 
\label{thm:characterization}
\end{theorem} 
\begin{proof} 
To ensure that $H$ is strictly positive definite for any $x_1,\ldots,x_n\in\mathbb{S}^{d-1}$, we show that 
\[
\phi(x^Ty) = \frac{1}{|\mathbb{S}^{d-1}|} \int_{\mathbb{S}^{d-1}} \gamma(x^Tw)\gamma(y^Tw) d w 
\]
is a strictly positive definite kernel on $\mathbb{S}^{d-1}$. For $d\geq 3$, the fact that $\gamma$ is Lipschitz continuous ensures that it has an absolutely and uniformly convergent ultraspherical expansion of the form~\cref{eq:ultraSexpansion}~\cite[Thm.~3]{xiang2020optimal}.  By the Funk--Hecke formula, we know that if $\gamma(t) = \sum_{k=0}^\infty a_k C_k^{((d-2)/2)}(t)$, then letting $\Gamma(\cdot)$ denote the Gamma function, we have the following absolutely convergent expansion for $\phi$: 
\[
\phi(x^Ty) = \! \sum_{k=0}^\infty c_k C_{k}^{((d-2)/2)}(x^Ty), \qquad  c_k = \frac{a_k^2}{b_{k,d}}\frac{\Gamma(d+k-2)}{\Gamma(d-2)\Gamma(k+1)},  
\]
where $b_{0,d} = 1$, $b_{1,d} = d$~\cite[Sec.~1.8]{gallier2009notes}, and 
\[
b_{k,d} = \binom{d+k-1}{k} - \binom{d+k-3}{k-2}, \qquad k\geq 2.
\]
From~\cite{xu1992strictly}, we find that $\phi(x,y) = \phi(x^Ty)$ is a strictly positive kernel on the sphere if and only if $\tilde{Z}_\phi = \{k\in\mathbb{Z}_+ : c_k>0\}$ contains infinitely many odd integers and infinitely many even integers. This completes the proof as $\tilde{Z}_\phi$ contains infinitely many odd integers and infinitely many even integers if $Z_\gamma = \{k\in\mathbb{Z}_+ : a_k\neq 0\}$ does. 
\end{proof} 

One can see that $\gamma(t) = \max(t,0)$ for $t\in[-1,1]$ is Lipschitz continuous on $[-1,1]$, but $\gamma(t) - t/2$ is an even function. This means that $\gamma(t) = \sum_{k=0}^\infty a_k C_k^{((d-2)/2)}(t)$ is an expansion with $a_{2k+1} = 0$ for $k\geq 1,$ so~\cref{thm:characterization} cannot be used to confirm that $H$ is strictly positive definite for all data sets. In fact,~\cref{def:ReLUHatesThis} gives an explicit data set that makes it singular. Moreover, if $\gamma$ is a Lipschitz continuous function on $[-1,1]$, then~\cref{thm:characterization} tells us that if $\gamma_{\text{even}}(t) =(\gamma(t) + \gamma(-t))/2$ and $\gamma_{\text{odd}}(t) =(\gamma(t) - \gamma(-t))/2$ are not polynomials, then the corresponding $H$ is strictly positive definite for all point sets on $\mathbb{S}^{d-1}$.

\subsection{Wendland kernels}
Motivated by our observations of the ReLU activation function and analysis in~\cref{sec:population,sec:finite}, we seek a choice of activation function $\gamma$ that ``optimizes'' two quantities: (1) The minimal eigenvalue of $H$ should be as large as possible, and (2) The condition number of $H$ should be as small as possible. The first point has immediate implications for the existence of solutions with zero training loss, while the second impacts both the actual training process and allows us to invoke classical results for robustness and stability. Towards this end, we advocate for the non-linearities derived from so-called Wendland kernels, which are compactly supported radial basis functions~\cite{wendland1995piecewise,wendland2004scattered}.

\subsubsection{The Kernels}
While there are infinitely many classes of Wendland kernels, we restrict our attention to the least smooth variants~\cite[Tab.~9.1]{wendland2004scattered}. Written for data on the sphere the two we consider are given by
\begin{equation}
\label{eqn:WendlandKernel}
\begin{aligned}
\phi_{d,0}(r) &= (1-r)_{+}^{\lfloor d/2\rfloor + 1}, \\
\phi_{d,2}(r) &= (1-r)_{+}^{\ell +2}[(\ell^2+4\ell+3)r^2+(3\ell + 6)r + 3],\quad \ell = \lfloor d/2\rfloor + 3,
\end{aligned}
\end{equation}
and the corresponding kernels are $\Phi_{d,k}(x,y) = \phi_{d,k}\!\left(\sqrt{2-2x^Ty}\right)$.  Notably, $\Phi_{d,k}$ is a strictly positive definite kernel on $\bbS^{d-1}$~\cite{wendland2004scattered}. To modulate the width of the kernels we introduce the parameter $\zeta >0$ and define $\Phi_{d,k,\zeta}(x,y) = \phi_{d,k}(\sqrt{2-2x^Ty}/\zeta).$ While all positive definite kernels must have eigenvalues that decay to zero as the number of data points grows, Wendland kernels have significantly slower decay than other common choices. This property is the impetus for our choice as we may expect that use of a feature map motivated by the Wendland kernels will lead to a well-conditioned $\gamma(X^TW)$ for $m$ not much larger than $n.$  

\begin{remark}
A rather interesting feature of~\cref{eqn:WendlandKernel} is that the form of the kernel depends on the dimension $d.$ This dependence is necessary to ensure the desired eigenvalue decay and maintain representational power as $d$ grows. Moreover, this is not a typical feature of activations functions used in the literature.
\end{remark}

\subsubsection{The activation functions}
We build an activation function from the Wendland kernel in the manner suggested by Reproducing Kernel Hilbert Spaces (RKHS). Specifically, we consider
\begin{equation}
\label{eqn:WendlandGamma}
\gamma(z) = \phi_{d,k,\zeta}(z).
\end{equation}
As before, it is important to note that using $\phi_{d,k,\zeta}$ as a non-linearity does not imply that the population level matrix $H$ is the associated kernel matrix. In fact, in general, we have
\[
H_{i,j} \neq \Phi_{d,k,1}(x_i,x_j)
\]
as the inner-product over the sphere induced by the uniform distribution is not the same inner-product associated with the RKHS constructed from $\phi_{d,k,\zeta}.$ Nevertheless,~\cref{thm:Wposdef} does show that $H$ is strictly positive definite and, therefore, when this activation function is used with distinct data points $\lambda_{\min}(H) > 0.$ 

\begin{theorem}
\label{thm:Wposdef}
Let $d\geq 3$, $\gamma(z) = \phi_{d,k,\zeta}(\sqrt{2-2z}/\zeta)$ and $x_1,\ldots,x_n\in\mathbb{S}^{d-1}$. Then, the $n\times n$ matrix
\[
H_{ij} = \frac{1}{|\mathbb{S}^{d-1}|} \int_{\mathbb{S}^{d-1}} \gamma(x_i^Tw)\gamma(w^Tx_j) dw, \qquad 1\leq i,j\leq n
\]
is strictly positive definite.
\end{theorem}
\begin{proof} 
Every Wendland kernel takes the form $\phi_{d,k,\zeta}(r) = \phi_{d,k}(r/\zeta) = p_{d,k}(r/\zeta)$ for $r\in[0,\zeta]$ and $0$ otherwise, where $p_{d,k}$ is a polynomial of degree $\lfloor d/2\rfloor + 3k + 1$ such that $\phi_{d,k,\zeta} \in \mathcal{C}^{2k}(\mathbb{R}_{\geq 0})$. Since $p_{d,k}(1) = 0$, we know that $p_{d,k}(r) = (r-1)q_{d,k}(r)$. We find that $\gamma(z)$ is Lipschitz continuous on $[-1,1]$. Looking at the even and odd part of $\gamma$, we find that 
\[
\frac{\gamma(z) \pm \gamma(-z)}{2} = \begin{cases}\frac{1}{2} p_{k,d}(\sqrt{2-2z}/\zeta),& z\in[\tfrac{2-\zeta^2}{2},1],\\ 0, & z\in[-\tfrac{2-\zeta^2}{2},\frac{2-\zeta^2}{2}],\\ \pm\frac{1}{2} p_{k,d}(\sqrt{2+2z}/\zeta),& z\in[-1,-\tfrac{2-\zeta^2}{2}]. \end{cases}
\]
When $\zeta<\sqrt{2}$, then the even and the odd part are not polynomials as they are zero on a set of positive measure, but not zero everywhere. For $\zeta = \sqrt{2}$, the even and the odd part are still not polynomial as $(\gamma(z)\pm\gamma(-z))/2$ are not infinitely differentiable at $z = 0$. Therefore, by~\cref{thm:characterization}, the matrix $A$ is strictly positive definite. 
\end{proof} 

There is often a rather stark contrast between the conditioning of $\gamma(X^TW)$ when either a ReLU or Wendland activation function is used (see~\cref{sec:Cond_numerical}). In general, the use of a Wendland activation function leads to a better conditioned matrix $\gamma(X^TW)$ once $m$ is a small multiple of $n$. While formally $\zeta$ should depend on properties of the data set $X$, we find that $\zeta = \sqrt{2}$ is a natural starting point as it gives the Wendland based activation function the same support as ReLU, i.e., it is non-zero if $x_i^Tx_j > 0$.

\subsection{Numerical experiments}\label{sec:Cond_numerical}

To complement our theoretical results and explore properties of~\cref{eqn:onelayer_model} with various activation functions, we provide several numerical experiments. We consider four data sets and four activation functions (see~\Cref{tab:activations}).
\begin{enumerate}
\item \emph{Synthetic-1}. We take the data points to be 500 i.i.d.\ points sampled from a uniform distribution on $\mathbb{S}^{9}$.
The label for a point is assigned to be $x_i$ is $x_i^Te$, where $e$ is the vector of all ones.\footnote{Here and elsewhere, ``labels'' is a generic term for the vector $y$ in \cref{eqn:LS_alpha}. 
For the Synthetic-1 and Prostate-Cancer datasets, the ``labels'' are outcomes or targets
in a regression problem.}
\item \emph{Synthetic-2}. We take the data points to be 92 i.i.d.\ points sampled from a uniform distribution on $\mathbb{S}^{2}$,
combined with the bad point set in \cref{def:ReLUHatesThis}, for a total of 100 points:
50 points are labeled $1$ (including all of the points in the bad point set),
and the other 50 points are labeled $-1$.
\item \emph{Prostate-Cancer}. We consider the 97 data points from 8 clinical measures (features) used to predict
prostate-specific antigen levels relevant to prostate cancer~\cite{stamey1989prostate}.
%\footnote{Data from \url{https://web.stanford.edu/~hastie/ElemStatLearn/datasets/prostate.data}.}
Each feature is transformed into its standard score, and then the data points are projected onto $\mathbb{S}^{7}$.
The labels are the log of the antigen levels.
\item \emph{Fashion-MNIST}. We generate the data points from $28 \times 28$ images of fashion items~\cite{xiao2017fashion}.
We first sample 250 images labeled as sneakers and 250 images labeled as t-shirts.
We project this subset onto its first 10 principal components and then project each embedding point onto $\mathbb{S}^{9}$.
The labels are $+1$ and $-1$ corresponding to sneakers and t-shirts, respectively.
\end{enumerate}

\begin{table}[tb]
\begin{center}
\caption{Activation functions used in numerical experiments. Here,
 $\mathbbm{1}_{(\mathcal{T})}$ is the indicator function, which equals $1$ if the logical expression $\mathcal{T}$ is true and $0$ otherwise.}\label{tab:activations}
\begin{tabular}{r l}
    \toprule
    name & activation function $\gamma(z)$ \\
    \midrule
     Wendland0 & $\gamma(z) = \mathbbm{1}_{(r \leq 1)}(1 - r)^{\ell},$ \\[1mm]
     		   &$r = \sqrt{2 - 2z} / \sqrt{2},$ and $\ell = \lfloor d / 2 \rfloor + 1$ \\[2mm]     		

     Wendland2 & $\gamma(z) =  \mathbbm{1}_{(r \leq 1)}(1 - r)^{\ell + 2}\left[(\ell^2 + 4\ell + 3)r^2 + (3\ell + 6)r + 3\right],$ \\[1mm]
     		   &$r = \sqrt{2 - 2z} / \sqrt{2},$ and $\ell = \lfloor d / 2 \rfloor + 3$\\[2mm]
     ReLU & $\gamma(z) = \max(z, 0)$ \\[2mm]
     swish & $\gamma(z) = z / \left(1 + e^{-z}\right)$ \\
     \bottomrule
\end{tabular}
\end{center}
\end{table}

We explore the condition number of $\hH$ for varying network widths and choices of the activation function. As a point of comparison, we provide information from $H$ either via an analytical expression or high accuracy quadrature. For quadrature, we use the Funk--Hecke formula paired with a numerically computed truncated ultraspherical polynomial expansion of $\gamma$. 
There are many reasons to consider the condition number when studying~\cref{eqn:onelayer_model}. First, we observe that a more ill-conditioned population matrix $H$ means that a larger width may be necessary for $\hH$ to approximate $H$ well (in the sense of \cref{lem:Happrox}).
Second, the conditioning of $\hH$ governs the convergence of gradient descent (see~\cref{sec:training}). Third, since the condition number of $\hH$ is the square of the condition number of the matrix in the least-squares problem in~\cref{eqn:LS_alpha}, which bounds the sensitivity of the solution to perturbations in the
labels~\cite[Chapter 21]{higham2002accuracy}.

In our experiments, we vary the network width $m$ and measure the condition number $\kappa(\hH)$,
where $\hH$ comes from the activation functions in~\cref{tab:activations}. In all of our experiments, we set $\zeta=\sqrt{2}$ in the Wendland kernels.
The swish activation is from a parametric family of activations
designed to be smooth approximations of ReLU~\cite{ramachandran2018searching};
here, we use the default implemented in TensorFlow~\cite{TF_swish}.
%\footnote{\url{https://www.tensorflow.org/api_docs/python/tf/nn/swish}.}\alex{Should the footnote really be a reference?}
%\austin{I think the docs are a better reference than the TF paper.}
For each width and activation function, we repeat the experiment 100 times with weights sampled i.i.d. from $\cU(\bbS^{d-1}).$

\begin{figure}[t]
\centering
\newcommand{\condfigwithfrac}{0.47}
\includegraphics[width=\condfigwithfrac\columnwidth]{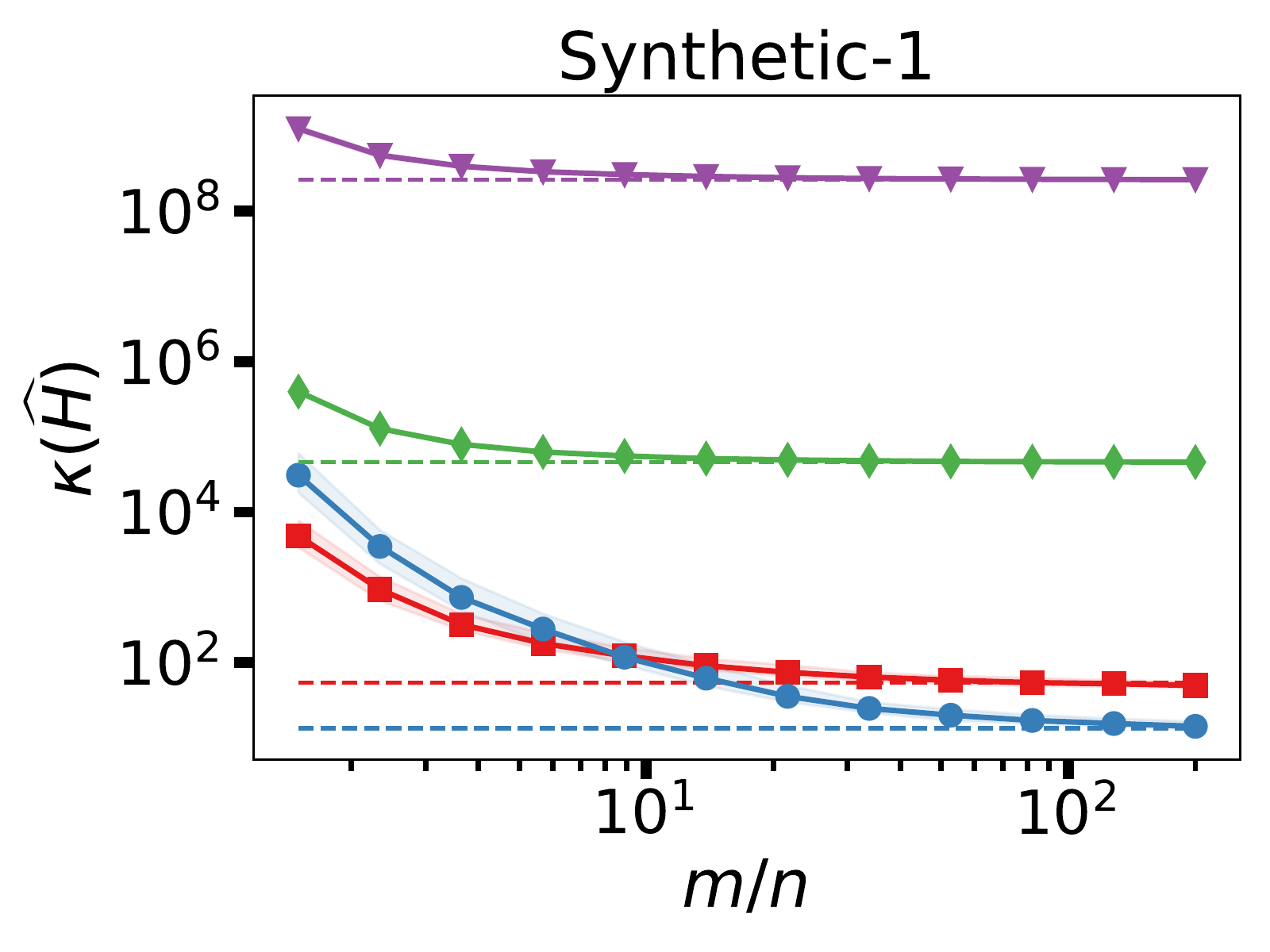}
\hfill
\includegraphics[width=\condfigwithfrac\columnwidth]{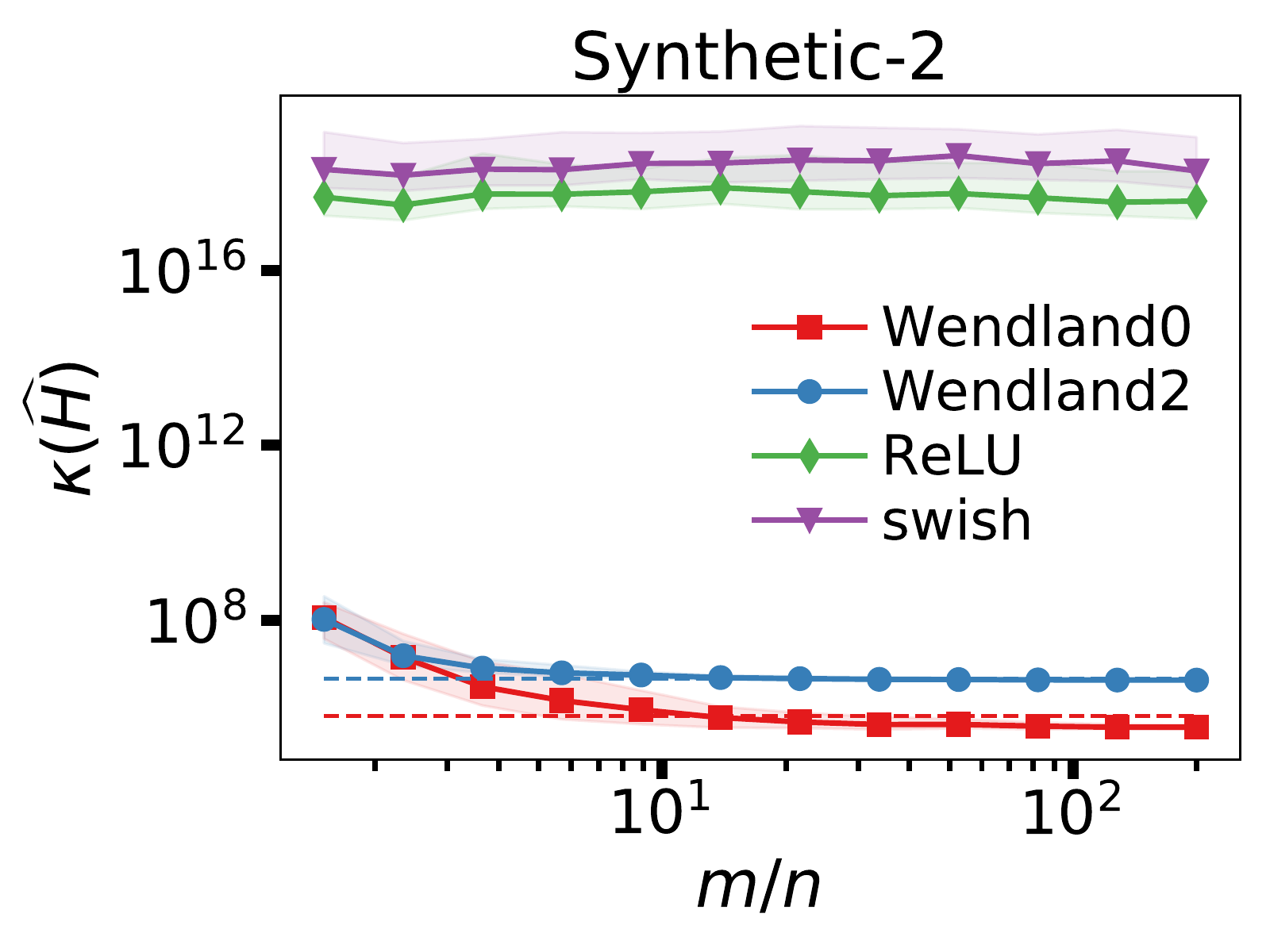} \\
\includegraphics[width=\condfigwithfrac\columnwidth]{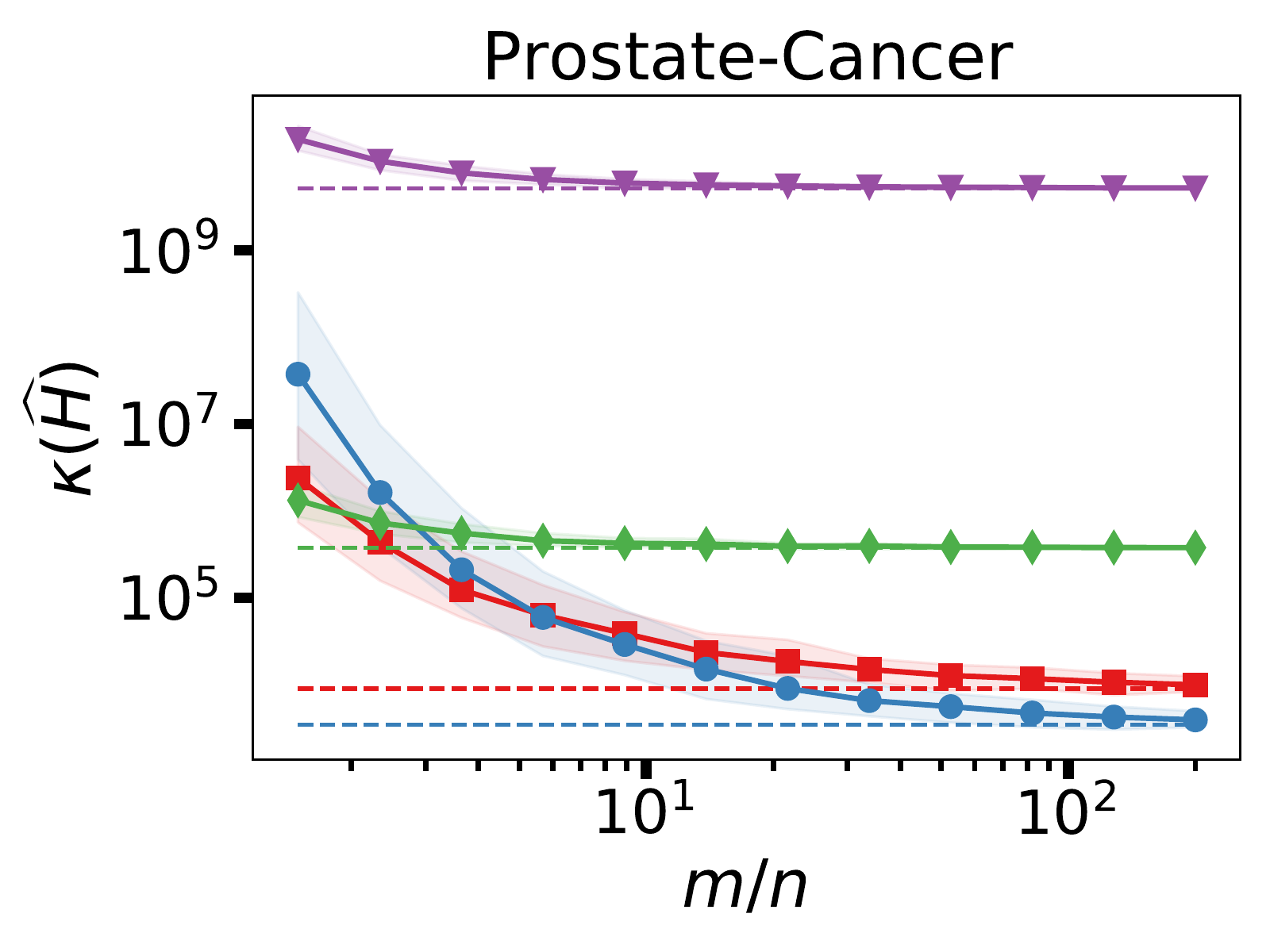}
\hfill
\includegraphics[width=\condfigwithfrac\columnwidth]{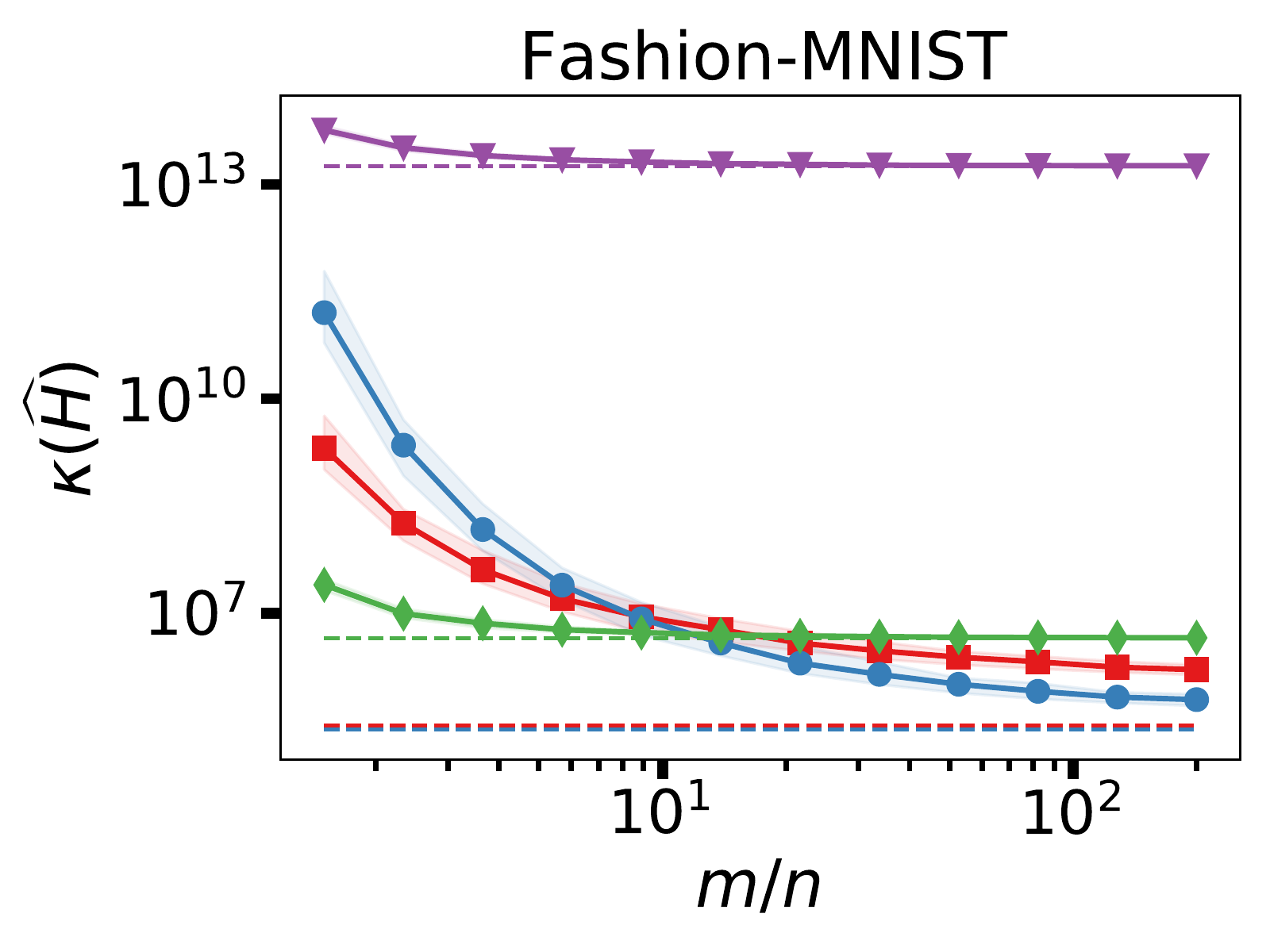}
\caption{Conditioning of \cref{eqn:LS_alpha} as a function of the width $m$ in four datasets.
Shading shows 10\%--90\% quantiles over 100 random initializations (markers are the median),
and the dashed lines are $\kappa(H)$, computed exactly for ReLU using \cref{eqn:ReLUfeature}
and numerically for the swish and Wendland activation functions.
The Synthetic-2 dataset contains the bad point set in \cref{thm:ReLUpoison}, which
renders the ReLU and swish matrices singular.
The Wendland activation functions lead to better conditioning as the width increases.
Often, $\kappa(\hH)$ approaches $\kappa(H)$ rapidly.
}
\label{fig:cond_width}
\end{figure}

As the network width becomes sufficiently large, the Wendland kernels lead to well-conditioned systems (see~\cref{fig:cond_width}). 
When using ReLU or swish, the conditioning shows little dependence on the network width,
and the condition number of $\hH$ converges to that of $H$ when $m \geq 10n$. 
Although the swish activation is essentially a smooth ReLU approximation~\cite{ramachandran2018searching}, we can see that the smoothness produces poorly conditioned matrices.
This aligns with the kernel perspective: generally, smooth kernels generate ill-conditioned kernel matrices.
The results on Synthetic-2 verify~\cref{thm:ReLUBad},
as $\hH$ is numerically singular for all widths.\footnote{As expected based on \cref{thm:ReLUpoison}, $H$ is numerically singular when its entries are computed explicitly.}
We also find that when the swish activation function is used $\gamma(X^TW)$ does not have full row-rank, which aligns with the results in \cref{sec:swish}.

%%%%%%%%%%%%%%%%%%%%%%%%%%%%%%%%%%%%%%%%
\section{Training loss}
\label{sec:training}
%%%%%%%%%%%%%%%%%%%%%%%%%%%%%%%%%%%%%%%%

In practice, how a network is trained is as much a part of the model as the network architecture itself. Modern deep networks are often vastly over parametrized, and the optimization algorithms used during training can have numerous hyper-parameters and nebulous stopping criteria. While some recent work~\cite{du2018gradient,du2019gradient} proves that zero training error is achievable for~\cref{eqn:onelayer_model} via gradient descent, the results are for unrealistically wide (large $m$) regimes and contain opaque consideration of the dependence on the minimal eigenvalue of $H$ (or, more appropriately, $G$).\footnote{As we saw previously, for finite $n$ and $d\geq 3$ this quantity can be made arbitrarily small,
requiring arbitrarily wide neural networks to reach zero training error.} 

\subsection{Zero training error in the random features regime}
One can characterize the set of simple neural networks~\cref{eqn:onelayer_model} with width $m\geq n$ that can achieve zero training loss for arbitrary $y$. This is not a new observation, as it is simply the theory of under-determined linear systems. 

\begin{proposition}
\label{prop:train_zero}
If there exists a $\hW\in\R^{d\times m}$ such that $\gamma(X^T\hW)$ is full row-rank, then 
\[
\min_{W,\alpha}\left\|\frac{1}{\sqrt{m}}\gamma(X^TW)\alpha - y\right\|_2^2 = 0
\]
for any $y\in\R^n.$
\end{proposition}
\begin{proof}
For any $A$ with full row-rank, $x = A^T(AA^T)^{-1}b$ satisfies $Ax=b.$
\end{proof}
In~\cref{prop:train_zero}, we have effectively fixed the training data $X.$ Furthermore, for fixed $\hW$ it is easy to write down an exact characterization of the minimal 2-norm solution $\halpha$ of~\cref{eqn:LS_alpha} (using, \emph{e.g.,} the SVD). However, practically, there are still numerous questions to address. 

\subsubsection{Ensuring full row-rank}
For any $\gamma$ and $X$ such that $\Hmin > 0$ we can easily ensure $\gamma(X^TW)$ is full row-rank by using~\cref{lem:evalue}; if
\[
m \geq \frac{10n\activbnd^2}{\Hmin}\log(2n/\delta)
\]
then $\gamma(X^TW)$ is full row-rank with probability $1-\delta.$\footnote{There is a broader class of $\gamma$ for which one can concoct $W$ based on the training data such that $\gamma(X^TW)$ is full row-rank even with $m=n$. We are uninterested in such schemes.} In~\cref{sec:activation}, the assumption that $\Hmin>0$ does not hold for ReLU. Nevertheless, for any fixed $X$ it is possible that $\Hmin > 0$ and we frame our results in terms of this quantity.

\subsection{Conditioning and the convergence of gradient descent}\label{sec:cond_gd}
Once we have a set of weights $W$ for which $\gamma(X^TW)$ is full row-rank, optimizing over the final-layer weights $\alpha$ reduces to the under-determined linear least-squares problem~\cref{eqn:LS_alpha}. To simplify the exposition in this section we let $Z \equiv \frac{1}{\sqrt{m}}\gamma(X^TW)$ be a fixed matrix.

While there are standard direct methods for solving such linear least-squares problems~\cite{GVL},
we consider what happens with gradient-based methods common in machine learning.

When $m\geq n$ and $Z$ has full row rank, gradient descent converges to a minimizer that achieves zero training error for a small enough fixed step size $\eta.$\footnote{If $m \leq n$ and $Z$ has full column rank, then~\cref{eqn:LS_alpha} can be analyzed as a Lipschitz-continuous convex function, and the condition number of $Z$ controls the complexity at which one can find a minimizer~\cite{nesterov1998introductory}. However, in this regime, generically it is not possible to achieve zero training loss.} First, we find it illustrative to show the convergence of gradient descent in~\cref{lem:GDconverge} explicitly as it both clearly shows the dependence of $\eta$ on the conditioning of $Z$ and that convergence occurs at the expected rate. \Cref{thm:NNconverge} then formalizes a statement about gradient descent, achieving zero training loss in our setting. Notably, in contrast to prior work, we have convergence to a specific solution (the least-norm one) that achieves zero training loss, not just convergence of the loss to zero.

\begin{lemma}
\label{lem:GDconverge}
Assume $Z\in\R^{n\times m}$ has full row-rank and consider~\cref{eqn:LS_alpha}. For any $\alpha^{(0)}$, the sequence of gradient descent iterates given by
\[
\alpha^{(k+1)} = \alpha^{(k)} - \eta (Z^TZ\alpha^{(k)} - Z^Ty), \qquad \eta = \frac{2}{\sigma_{\min}^2(Z)+\sigma_{\max}^2(Z)},
\]
satisfies
\[
\|Z\alpha^{(\ell)} - y \|_2 \leq \left(1-\frac{2}{1+\kappa(Z)^2}\right)^{\ell}\|Z\alpha^{(0)}-y\|_2. 
\]
\end{lemma}
\begin{proof}
First, we relate the objective function at step $k+1$ to step $k$ as
\begin{align*}
\|Z\alpha^{(k+1)} - y\|_2 &= \|Z(\alpha^{(k)} - \eta (Z^TZ\alpha^{(k)} - Z^Ty)) - y\|_2 \\
&= \|Z\alpha^{(k)} - y - \eta ZZ^T(Z\alpha^{(k)} - y) \|_2 \\
&= \|(I - \eta ZZ^T)(Z\alpha^{(k)} - y) \|_2 \\
&\leq \| I - \eta ZZ^T\|_2\|Z\alpha^{(k)} - y\|_2.
\end{align*}
This immediately implies that
\begin{align}
\|Z\alpha^{(\ell)} - y\|_2 \leq \| I - \eta ZZ^T \|_2^\ell \|Z\alpha^{(0)} - y\|_2. \label{eq:gd_err_bound}
\end{align}
All that remains to ensure convergence, i.e., $\|Z\alpha^{(\ell)} - y\|_2 \rightarrow 0$ as $\ell \rightarrow \infty$, is to pick an $\eta$ sufficiently small so that $\|(I - \eta ZZ^T)\|_2 < 1.$ To accomplish this, we minimize
\[
\| I - \eta ZZ^T \|_2 = \max\!\left\{1-\eta \sigma_{\min}^2(Z),\eta \sigma_{\max}^2(Z)-1\right\},
\]
where $\sigma_{\min}(Z)$ and $\sigma_{\max}(Z)$ are the smallest and largest singular values of $Z.$ Concretely, the minimizer occurs when $1 - \eta \sigma_{\min}^2(Z) = \eta\sigma_{\max}^2(Z) - 1$ so that $\eta = 2/(\sigma_{\min}^2(Z) + \sigma_{\max}^2(Z))$.
%\begin{equation}\label{eq:opt_step_size}
% \Longrightarrow  \eta = \frac{2}{\sigma_{\min}(Z)^2 + \sigma_{\max}(Z)^2}.
%\end{equation}
Using this choice of $\eta$, we find that
\[
\| I - \eta ZZ^T \|_2 = 1 - \frac{2\sigma_{\min}^2(Z)}{\sigma_{\min}^2(Z) + \sigma_{\max}^2(Z)} = 1-\frac{2}{1+\kappa(Z)^2} < 1,
\]
and the result follows.
\end{proof}

When $m > n$, there are an infinite number of global minimizers for~\cref{eqn:LS_alpha}, and
initializing with $\alpha^{(0)} = 0$ is a sufficient condition for
converge to the minimizer $\alpha^{\textnormal{ln}}$ of \cref{eqn:LS_alpha}
with the minimum 2-norm. Similar results hold for certain stochastic gradient
methods~\cite{ma2015convergence,strohmer2009randomized}.

Building on \cref{lem:GDconverge}, \cref{thm:NNconverge} provides a clear statement of when gradient descent applied to the single hidden layer model \cref{eqn:onelayer_model} can achieve zero training loss.
\begin{theorem}
\label{thm:NNconverge}
Consider
\begin{itemize}
\item a dataset $\{(x_i,y_i)\}_{i=1}^n$ with $x_i\in\R^d,$ $y_i\in\R,$ $\|x_i\|_2=1$ for all $i = 1, \ldots, n$, and $x_i\neq x_j$ for $i\neq j$;
\item $\gamma$ and $H$ as in~\cref{eqn:H} with $\gamma$ ensuring that $\Hmin>0$; and
\item a failure probability $0 < \delta < 1$.
\end{itemize}
Let 
\[
m \geq 10n\activbnd^2 \log(2n/\delta) /\Hmin
\]
and $W \in \mathbb{R}^{d \times m}$ be a weight matrix with columns 
drawn i.i.d.\ from $\cU(\bbS^{d-1})$. Then, there is a unique least-norm solution to~\cref{eqn:LS_alpha} that achieves zero objective value with probability at least $1-\delta$.
%\begin{equation}
%\label{eqn:minnorm}
%\min_{\alpha} \|\alpha\|_2 \quad \emph{s.t.} \quad \left\|\frac{1}{\sqrt{m}}\gamma(X^TW)\alpha - y\right\|_2 = 0
%\end{equation}
%has a unique solution with probability at least $1-\delta$. 
Furthermore, when there is a unique least-norm solution the sequence of iterates $\{\alpha^{(k)}\}_{k=0}^{\infty}$ generated by gradient descent on the function
\[
f(\alpha) = \frac{1}{2}\left\|\frac{1}{\sqrt{m}}\gamma(X^TW)\alpha - y\right\|_2^2 
\]
with $\alpha^{(0)}=0$ and step size $\eta$ satisfying $0 < \eta < 2/\lambda_{\max}(\hH)$ converges to the least-norm solution of~\cref{eqn:LS_alpha} as $k\rightarrow \infty$, where $\hH$ is given in~\cref{eq:Hhat}.
\end{theorem}
\begin{proof}
First, by~\cref{lem:evalue} we have that $\hH\succ 0$ and therefore $\gamma(X^TW)$ is full row-rank. This implies that~\cref{eqn:LS_alpha} has a unique solution. Now, we let
\[
\frac{1}{\sqrt{m}}\gamma(X^TW) = U\Sigma V^T
\]
be the reduced SVD of $\frac{1}{\sqrt{m}}\gamma(X^TW)$ where $U\in\R^{n\times n}$ and $V\in\R^{m\times n}$ have orthonormal columns and $\Sigma\in\R^{n\times n}$ is diagonal with entries $\sigma_j > 0$ for $1\leq j\leq n$. Starting with $\alpha^{(0)} = 0$, by \cref{lem:Landweber}, we have that
\begin{align*}
\alpha^{(\ell)} = Vf_\ell(\Sigma)\Sigma^{-1}U^Ty,
\end{align*}
where $f_\ell(z) = 1 - \left(1 - \eta z^2\right)^\ell$. Using the fact that the unique minimal norm solution to~\cref{eqn:LS_alpha} is $\alpha^{(\text{ln})} = V\Sigma^{-1}U^Ty$, we have that
\begin{equation*}
\begin{aligned}
\|\alpha^{(\text{ln})} - \alpha^{(\ell)}\|_2 &= \| V\Sigma^{-1}U^Ty - Vf_\ell(\Sigma)\Sigma^{-1}U^Ty\|_2 \\
&\leq \| I - f_\ell(\Sigma)\|_2\|\Sigma^{-1}\|_2\|y\|_2 \\ 
&= \|\Sigma^{-1}\|_2\|y\|_2  \max_{1\leq j\leq n} \, \lvert1 - \eta \sigma_j^2\rvert^\ell. 
\end{aligned}
\end{equation*} 
Since $\sigma_j^2=\lambda_j(\hH)$ for $1\leq j\leq n$ choosing $0 < \eta < 2/\lambda_{\max}(\hH)$ ensures that $\max_{1\leq j\leq n} \, \lvert1 - \eta \sigma_j^2\rvert < 1$ and the result follows.
\end{proof}

\begin{remark}
The optimal step size in \cref{thm:NNconverge} (with respect to an upper bound on convergence to the least-norm solution) is $\eta = 2/\left(\lambda_{\min}(\hH)+\lambda_{\max}(\hH)\right).$ In this case we have that
\[
\|\alpha^{(\text{ln})} - \alpha^{(\ell)}\|_2 \leq \|\Sigma^{-1}\|_2\|y\|_2  \left(1-\frac{2}{1+\kappa(\hH)}\right)^\ell.
\]
\end{remark}

We may compare this result to aforementioned theoretical results on achieving zero training loss by training $W.$ Interestingly, the lower bounds in \cref{thm:NNconverge} are considerably smaller. For example, to ensure convergence to zero training loss with probability $1-\delta$ \cite[Theorem 5.1]{du2019gradient} requires 
\[
m = \Omega\left(\max\left\{\frac{n^4}{\lambda_{\min}^4(G)},\frac{n}{\delta},\frac{n^2\log (n/\delta)}{\lambda_{\min}^2(G)}\right\}\right),
\]
where $G$ is the NTK matrix for the chosen activation function. Moreover, these results require smaller step sizes to ensure convergence. 

Irrespective of the known theoretical results, one may expect that~\cref{thm:NNconverge} could be improved if we simultaneously optimize over $W$ and $\alpha$ since $\gamma(W^TX)$ may become better conditioned as the weights are ``learned''. In~\cref{sec:joint}, we discuss this possibility and, in practice, observe the opposite effect. The condition number of $\gamma(W^TX)$ is observed to grow during training. This may suggest that zero training loss is achieved more quickly by simply fixing the weights at their initialization.

\subsection{Numerical experiments}\label{sec:GD_numerical}
Using the datasets from~\cref{sec:Cond_numerical}, we explore the implications of \cref{lem:GDconverge} and \cref{thm:NNconverge}.
Specifically, we demonstrate that gradient descent can eventually converge to zero training error and illustrate the dependence of that behavior on the conditioning of $\gamma(X^TW).$

\begin{figure}[t]
\centering
\newcommand{\condfigwithfrac}{0.47}
\includegraphics[width=\condfigwithfrac\columnwidth]{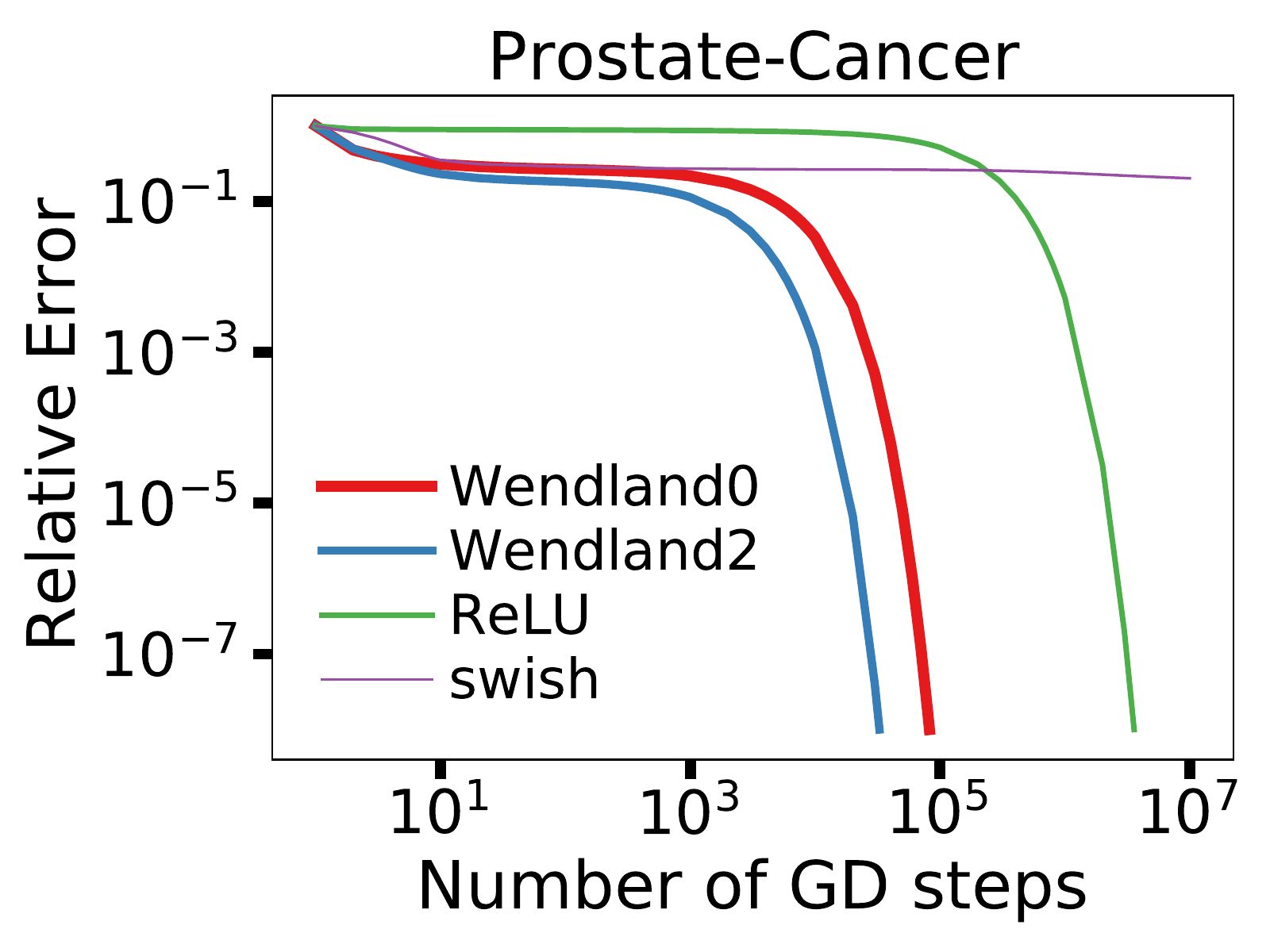}
\hfill
\includegraphics[width=\condfigwithfrac\columnwidth]{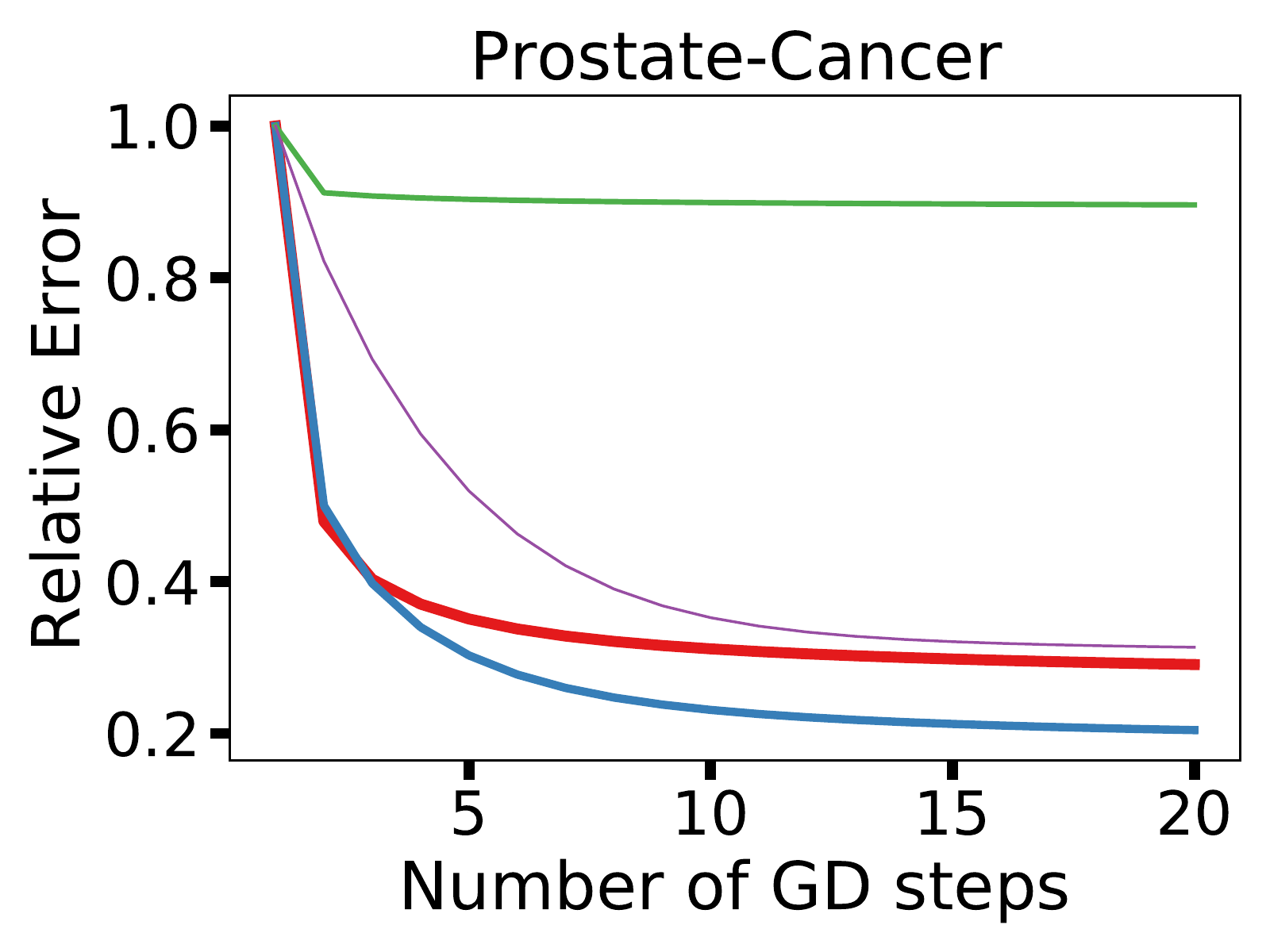}
\caption{Optimization of \cref{eqn:LS_alpha} with gradient descent with width $m = 200n$ on the Prostate-Cancer dataset
using the step size in~\cref{lem:GDconverge}.
The condition number drives when error rapidly decreases (left),
even though early steps have different behavior (right).
Early steps can be interpreted via Landweber iteration (see \cref{sec:landweber}).}
\label{fig:gd_training}
\end{figure}

First, we using the Prostate-Cancer dataset with $m = 200n$. \Cref{fig:gd_training} shows the dependence between the condition number and the convergence rate of gradient descent. We observe that the condition number of $\kappa(\hH)$ is smaller for the Wendland activation functionss (see~\cref{fig:cond_width}).
Indeed, gradient descent with the fixed step size given in~\cref{lem:GDconverge} shows that the Wendland activation functions lead to a faster convergence rate than ReLU and swish.

At the same time, the number of gradient descent steps needed for convergence is quite large\emthin tens of thousands
of steps for the Wendland activation and over a million steps for ReLU. With the swish activation function, gradient descent does not even converge after 10 million steps. This is far larger than the number of steps one might use in practice~\cite{zhang2017understanding}, and the decrease in error stagnates after just 20 iterations (see~\cref{fig:gd_training}, right). % The error remains at this stagnated level until the optimization method is close enough to the minimizer to achieve rapid convergence. Thus, we explore what is happening in early iterations of the optimization procedure.

\subsection{Early iteration behavior as Landweber iteration}\label{sec:landweber}
To explain the differences in early iterations of gradient descent during training, we investigate the Landweber iteration~\cite{hansen2010discrete}.
As before, denote $\frac{1}{\sqrt{m}}\gamma(X^TW)$ by $Z$.
Let $Z = U\Sigma V^T$ be the ``reduced'' SVD of $Z$, where $U\in\R^{n\times n},$ $\Sigma\in\R^{n \times n}$, and $V\in\R^{m\times n}$. \Cref{lem:Landweber} shows that the iterates of gradient descent for~\cref{eqn:LS_alpha} with fixed step size $\eta$ are given by $\alpha^{(k)} = Vf_k(\Sigma)\Sigma^{-1}U^Ty$, where
\begin{align}
f_k(z) = 1 - \left(1 - \eta z^2\right)^k. \label{eqn:filter}
\end{align}
Notably, taking $\eta$ as in~\cref{lem:GDconverge} we have that
\begin{equation}
\begin{aligned}
\label{eqn:SigmaFilter}
f_k(\sigma_j) &= 1 - \left(1 - \eta \sigma_j^2\right)^k \\
&= 1 - \left(1 - \eta \lambda_j(\hH)\right)^k\\
&= 1 - \left(1 - \frac{2\lambda_j(\hH)}{\lambda_{\min}(\hH) + \lambda_{\max}(\hH)}\right)^k.
\end{aligned}
\end{equation}
Moreover, we also have that
\begin{equation}
\begin{aligned}
\label{eqn:trainingdecay}
\|Z\alpha^{(k)} - y\|_2 &= \|ZV(I-(I-\eta\Sigma^2)^k)\Sigma^{-1} U^Ty - y\|_2 \\
&= \|U\Sigma(\Sigma^{-1}-(I-\eta\Sigma^2)^k\Sigma^{-1}) U^Ty - y\|_2\\
&= \|(I-\eta\Sigma^2)^kU^Ty\|_2\\
&=\left(\sum_{j=1}^n u_j^Ty (1-\eta\sigma_j^2)^{2k}\right)^{1/2}.
\end{aligned}
\end{equation}

The key observation is that if $f_k(\sigma_j) \approx 1$ for all the singular values then $\alpha^{(k)}\approx V\Sigma^{-1}U^Ty,$ which is the least-norm solution of~\cref{eqn:LS_alpha}. This is accomplished for small $k$ if the eigenvalues of $\hH$ are relatively flat such that $2\lambda_j(\hH) \approx \lambda_{\min}(\hH) + \lambda_{\max}(\hH)$ for all $j.$ We empirically observe that the desired slow eigenvalue decay from the Wendland activation functions (see~\cref{fig:landweber} (left)). \Cref{fig:landweber} (right) shows that many Landweber filter values $f_k(\sigma_j) \approx 1$ by $k = 20.$ The faster eigenvalue decay seen when using ReLU or swish explain the relative benefits of the Wendland activation function during the early iterations. 

A more nuanced view of this discussion is that gradient descent quickly ``resolves'' certain components of the solution that are associated with the components of $y$ in the direction of left singular vectors corresponding to singular values close to 1. This is further illustrated by~\cref{eqn:trainingdecay} and \cref{lem:Landweber}, where we explicitly see the interplay between singular values, the step size, and the decomposition of $y$ in the basis of left singular vectors.

\begin{figure}[ht!]
\centering
\newcommand{\condfigwithfrac}{0.47}
\includegraphics[width=\condfigwithfrac\columnwidth]{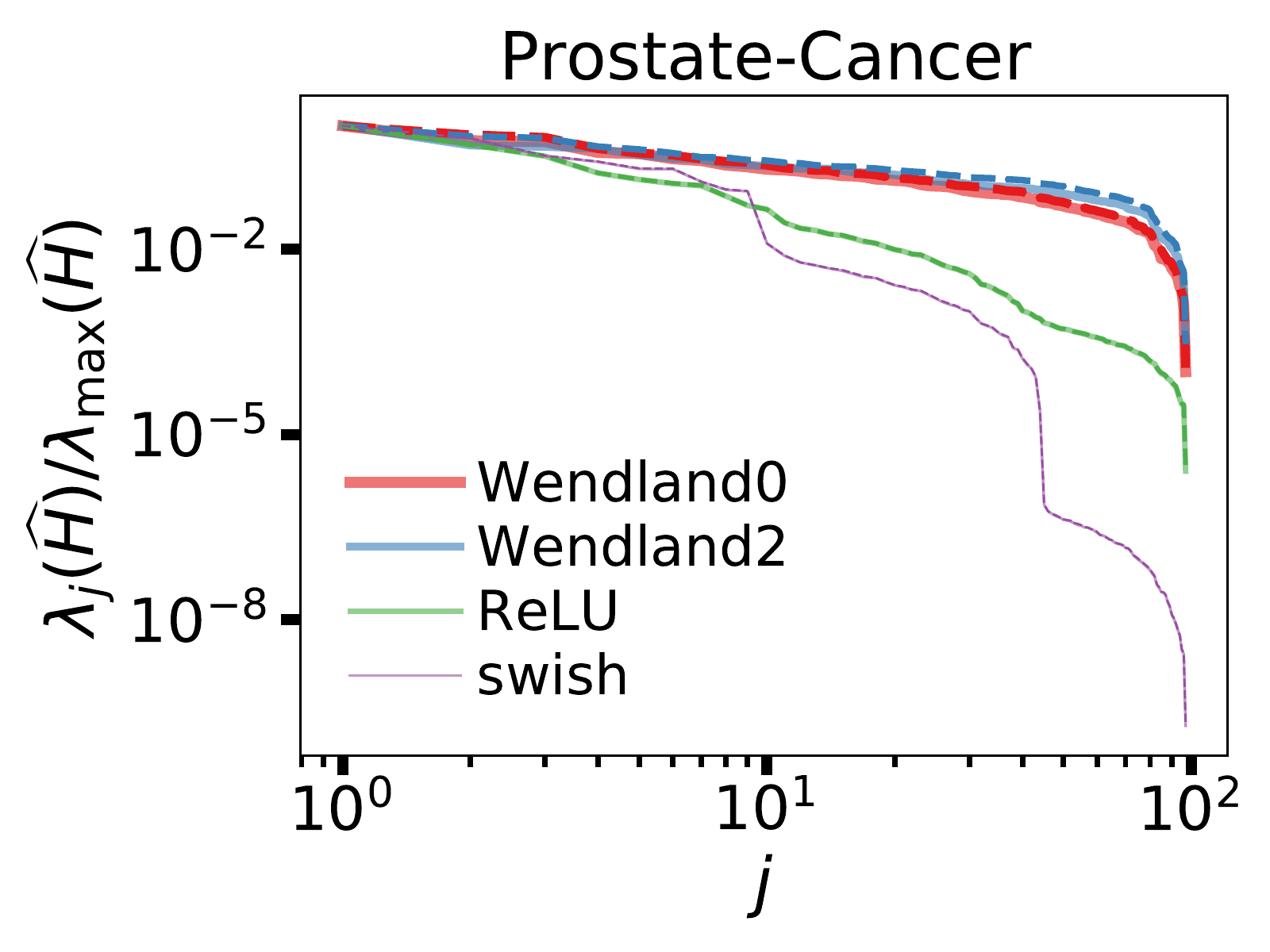}
\hfill
\includegraphics[width=\condfigwithfrac\columnwidth]{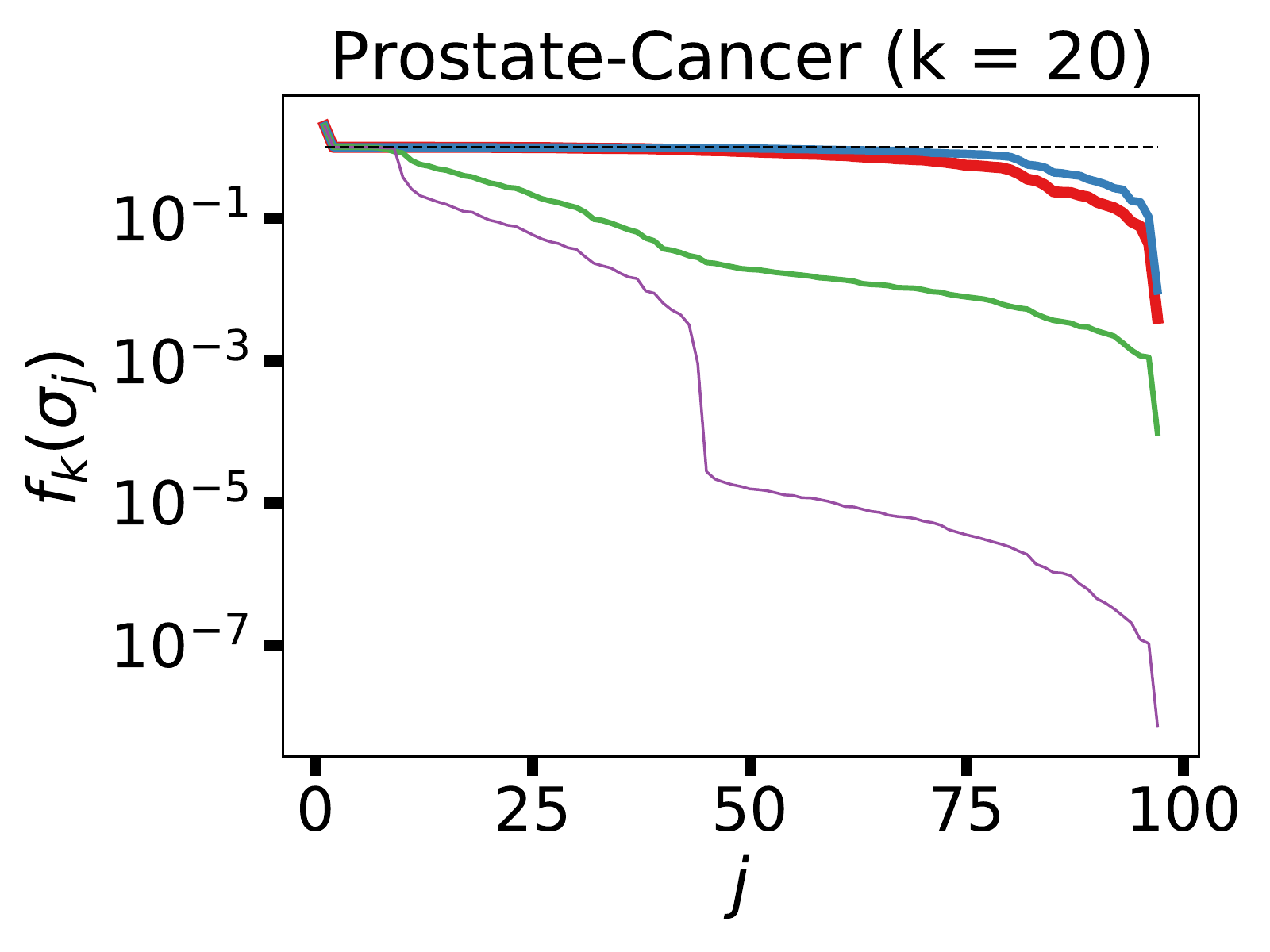}
\caption{(Left) The eigenvalues of $\hH$ (solid lines) and $H$ (dashed lines), which are flat for the Wendland kernel.
(Right) The Landweber filter values from \cref{eqn:filter} are near 1 (dashed black line) for the Wendland activation functions
because of the slow spectral decay. Thus, early iterates are close to the least-norm solution of \cref{eqn:LS_alpha}.}
\label{fig:landweber}
\end{figure}

% 1. It is ``easy'' to get zero training error without weight training (effectively no gap between the random features and neural tangent regimes, from a training point of view)
% 2. Networks with many common activation functions will struggle to get zero training error (including ReLU)
% 3. Our setup makes it easier to understand optimization

%%%%%%%%%%%%%%%%%%%%%%%%%%%%%%%%%%%%%%%%
\section{Conclusions}
\label{sec:conclusions}
%%%%%%%%%%%%%%%%%%%%%%%%%%%%%%%%%%%%%%%%
Our theory and experiments highlight the limits of certain common analysis techniques and metrics, advocate for careful consideration of activation functions, and provide insight on properties that aid in understanding the training of certain simple neural networks. By considering the random features regime we were able to extensively leverage existing theory for under-determined linear system and, from this perspective, it is entirely unsurprising that simple optimization methods are able to find solutions that achieve zero-training loss.
Thus, careful understanding of the properties of models learned during training beyond their training loss is essential to discern the relative differences between models. For simple over-parameterized models, cleanly characterizing sets of solutions that achieve small training error and have preferable properties (\emph{e.g.,} favorable robustness or generalization behavior) provides a major opportunity.

\section*{Acknowledgements}
We thank David Bindel, Chris De Sa, Andrew Horning, and Victor Minden for valuable discussion and feedback.  
This research was supported in part by 
NSF Award DMS-1830274, NSF Award DMS-1818757,
ARO MURI, 
ARO Award W911NF19-1-0057, FACE Foundation,
and JP Morgan Chase \& Co.

\bibliographystyle{siam}
\bibliography{refs}

\appendix

\section{Supplementary Material}
\label{sec:appendix}
%%%%%%%%%%%%%%%%%%%%%%%%%%%%%%%%%%%%%%%%
%\anil{We can put supplemental info here for the moment and pull it out into a separate document later.}

\subsection{Additional bounds for Gaussian weights}
\label{sec:NormalBounds}
Throughout the paper we assumed that the weights (entries of $W$) are uniformly drawn i.i.d.\ from the sphere.
Here, we show that an analogous result to \cref{thm:NNconverge} holds if the
weights are normally distributed. We omit a discussion of properties of $H$ (note that its definition must change to respect the change in how $W$ is constructed) in this setting and simply frame everything in terms of $\Hmin$ assuming $\Hmin > 0.$ 

Let $W_{ij} \sim \mathcal{N}(0,1)$ be i.i.d.
In this setting, the population level matrix has entries given by 
\begin{equation}
H_{i,j} = \mathbb{E}_{w \sim \mathcal{N}(0, I)} \left[\gamma(x_i^Tw)\gamma(w^Tx_j) \right]. \label{eq:H_normal}
\end{equation}
Again, $H$ is a Gram matrix, so $H \succeq 0$ for all continuous $\gamma.$
In the finite width case, with width $m$, we have
\begin{align}
\hH_{ij} = \frac{1}{m}\sum_{k=1}^{m}\gamma(x_i^Tw_k)\gamma(w_k^Tx_j) = \frac{1}{m}\sum_{k=1}^{m}h_kh_k^T. \label{eq:Hhat_normal}
\end{align}
where $h_k = \gamma(X^Tw_k)^T$.

We can again answer how well $\hH$ approximates $H$.
To do so, we first slightly expand our assumption on the activation function $\gamma,$ we now require that
\begin{equation}
\max_{x \in \mathbb{R}} \,\lvert \gamma(x) \rvert \leq \activbnd \lvert x \rvert. \label{eq:gamma_bound_normal}
\end{equation}
For example, \cref{eq:gamma_bound_normal} is satisfied with $\activbnd = 1$ for the ReLU and swish activation functions.

Next, the theoretical setup will depend on sub-Gaussian random variables.
\begin{definition}[Sub-Gaussian random vector]
A random vector $y$ in $\mathbb{R}^d$ is sub-Gaussian with parameter $J$ if 
\begin{equation}
\prob\{ \lvert y^Tz \rvert \geq t\} \le 2e^{-t^2/J^2} \text{ for all $z \in \mathbb{S}^{d-1}$}.
\end{equation}
\end{definition}
The $h_k$ in \cref{eq:Hhat_normal} are sub-Gaussian random vectors.
\begin{lemma}
Let $W$ have i.i.d.\ $\mathcal{N}(0,1)$ entries.
Then, $h_k = \gamma(X^Tw_k)$ is sub-Gaussian with parameter $J = c\activbnd\sigma_1(X)$, where $c$ is an absolute constant
and $\activbnd$ is from \cref{eq:gamma_bound_normal}.
\end{lemma}
\begin{proof}
For any $z' \in \bbS^{n-1}$,
\[
\prob\{\lvert h_r^Tz' \rvert \geq t\} \leq 
\sup_{z \in \bbS^{n-1}} \prob\{ \lvert h_r^Tz \rvert \geq t\} \leq 
\sup_{z \in \bbS^{n-1}} \prob\{ \lvert \activbnd y_r^Tz \rvert \geq t\},
\]
where the second inequality follows from \cref{eq:gamma_bound_normal}.
Since $W$ has i.i.d.\ $\mathcal{N}(0,1)$ entries and each data point in $X$ has unit 2-norm, $y_r = X^Tw_r \sim \mathcal{N}(0, XX^T)$
and $\activbnd y_r^Tz \sim \mathcal{N}(0, \activbnd^2z^TXX^Tz)$.
Thus,
\[
\sup_{z \in \bbS^{n-1}} \prob( \lvert y_r^Tz \rvert \geq t) \leq 
c \cdot \sup_{z \in \bbS^{n-1}} \sqrt{\activbnd^2z^TXX^Tz} = 
c\activbnd \cdot \sup_{z \in \bbS^{n-1}} \| X^Tz \| = c\activbnd \sigma_1(X).
\]
The first inequality follows from the fact that a univariate normal random variable is sub-Gaussian with sub-Gaussian norm
equal to an absolute constant times its variance~\cite[Example 2.5.8]{vershynin2018high}.
\end{proof}

We use the following bound on the second moments of sub-Gaussian random variables.
\begin{proposition}
Let $Z_1, \ldots, Z_m$ be i.i.d.\ random vectors in $\mathbb{R}^n$ that have sub-Gaussian
distributions with parameter $J$ and let $m \geq n$. Then,
\[
\left\| \frac{1}{m}\sum_{i=1}^{m}Z_iZ_i^T - \mathbb{E}Z_{i}Z_{i}^T \right\|_2 \le \epsilon
\]
with probability at least $1 - 2e^{2n - c'm\epsilon^2 / J^2}$, where $c'$ is an absolute constant.
\end{proposition}
\begin{proof} 
See the proof of Proposition 2.1 in~\cite{vershynin2012close}.
\end{proof} 
An immediate corollary is a lower bound on the width that makes $\hH$ close to $H$.
\begin{corollary}[Analog of \cref{lem:Happrox} for normally distributed weights]
Let $0 < \delta < 1$. There is an absolute constant $C$, where, if 
\begin{equation}
m \geq \frac{C(2n + \log(2/\delta))\activbnd^2}{\Hmin}  \left(\frac{\sigma_1(X)^2}{\Hmin}\right), \label{eq:widthbound_normal}
\end{equation}
then $\Pr\{\|\hH-H\|_2 < \frac{\Hmin}{2}\} > 1-\delta$.
\end{corollary}
In this result, $\sigma_1(X)^2/\Hmin$ is the analog of $\kappa(H)$ in~\cref{lem:Happrox}.

We can use this fact to develop an analogous result for zero training error with
weights initialized as i.i.d.\ Gaussian entries.
\begin{theorem}[Analog of  \cref{thm:NNconverge}]
\label{thm:NNconverge_normal}
Consider
\begin{itemize}
\item a dataset $\{(x_i,y_i)\}_{i=1}^n$ with $x_i\in\R^d,$ $y_i\in\R,$ $\|x_i\|_2=1$ for all $i = 1, \ldots, n$, and $x_i\neq x_j$ for $i\neq j$;
\item $\gamma$ and $H$ as in~\eqref{eq:gamma_bound_normal}~and~\eqref{eq:H_normal} with $\gamma$ ensuring that the $\Hmin>0$; and
\item a failure probability $0 < \delta < 1$.
\end{itemize}
Let $m$ satisfy the bound in~\eqref{eq:widthbound_normal} and $W \in \mathbb{R}^{d \times m}$ be a weight matrix with 
i.i.d.\ random entries $W_{ij} \sim \mathcal{N}(0, 1)$. Then, there is unique least-norm solution of~\cref{eqn:LS_alpha} that achieves zero training loss with
with probability at least $1-\delta$. 
Furthermore, the sequence of iterates $\{\alpha^{(k)}\}_{k=0}^{\infty}$ generated by gradient descent on the function
\[
f(\alpha) = \frac{1}{2}\left\|\frac{1}{\sqrt{m}}\gamma(X^TW)\alpha - y\right\|_2^2 
\]
with $\alpha^{(0)}=0$ and step size $\eta = 2/\left(\lambda_{\min}(\hH)+\lambda_{\max}(\hH)\right)$ converges to the least-norm solution of~\cref{eqn:LS_alpha} as $k\rightarrow \infty$, where $\hH$ is given in~\eqref{eq:Hhat_normal}.
\end{theorem}

\subsection{Failure of the swish activation function}  
\label{sec:swish}
Here, we show that the swish activation function has the same issues as ReLU (see~\cref{thm:ReLUpoison}) with
the same point set (see~\cref{def:ReLUHatesThis}).
The swish feature map is $\gamma(t) = t/(1+e^{-t})$. 
Note that $(\gamma(t) + \gamma(-t))/2 = z/2$ so that $\gamma(t) = \gamma_{\text{even}}(t) + t/2$, i.e., 
\[
\gamma(t) = \sum_{k=0}^\infty a_{2k}C_{2k}^{((d-2)/2)}(t) + \frac{1}{2(d-2)}C_1^{((d-2)/2)}(t). 
\] 
From the proof of~\cref{thm:characterization} we find that
\[
\begin{aligned} 
\phi(x^Ty) &= \frac{1}{|\mathbb{S}^{d-1}|}\int_{\mathbb{S}^{d-1}}\!\! \gamma(x^Tw)\gamma(y^Tw)dw\\
& = \sum_{k=0}^\infty c_{2k}C_{2k}^{((d-2)/2)}(t) +  \frac{1}{4d(d-2)}C_1^{((d-2)/2)}(t)\\
& = \underbrace{\sum_{k=0}^\infty c_{2k}C_{2k}^{((d-2)/2)}(t)}_{\gamma_{\text{even}}(t)} +  \frac{t}{4d}. 
\end{aligned} 
\]
Therefore, we cannot conclude that $H$ is a strictly positive definite matrix.
In fact, we find that the point set in~\cref{eq:ReLUHatesThis} also causes the matrix $H_{jk} = \phi(x_j^Tx_k)$ to be singular. To see this, one can repeat the argument in~\cref{thm:ReLUpoison}. In the same way we need to show that $\phi(1) - 3\phi(s^2) + 3\phi(-s^2) - \phi(-1) = 0$ when $s =1/\sqrt{3}$, which holds since $\phi_{\text{even}}(1) - 3\phi_{\text{even}}(s^2)  + 3\phi_{\text{even}}(-s^2)  - \phi_{\text{even}}(-1)  = 0$ and hence,
\[
\phi(1) - 3\phi(s^2) + 3\phi(-s^2) - \phi(-1) = \frac{1}{4d}  - \frac{1}{4d} - \frac{1}{4d} - + \frac{1}{4d} = 0. 
\]

\subsection{Joint weight training}\label{sec:joint}

\begin{figure}[t]
\centering
\includegraphics[width=0.5\columnwidth]{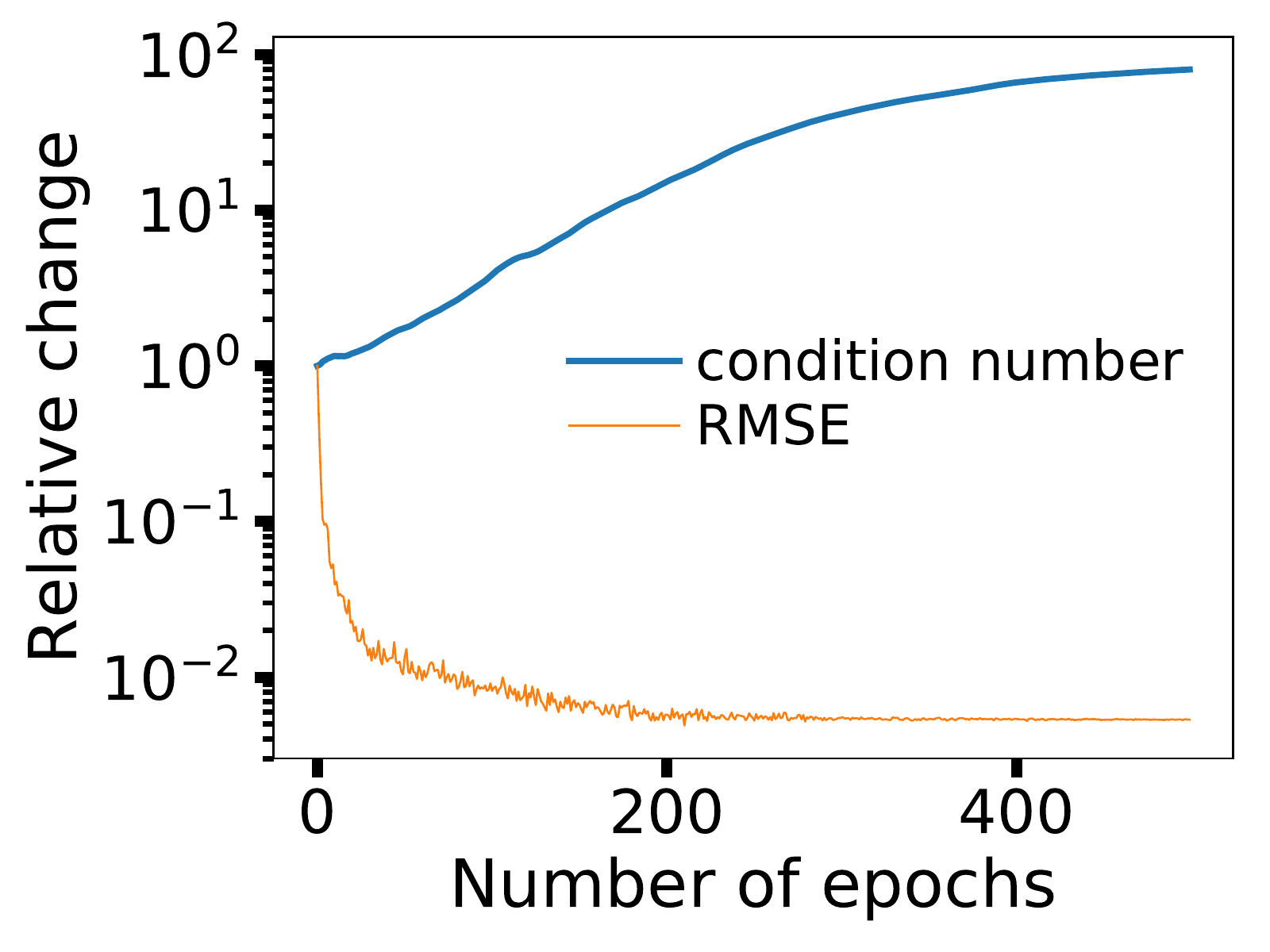}
\caption{Condition number of $\gamma(W^TX)$ using ReLU during joint training of the
hidden layer weights $W$ and the final layer $\alpha$ on the Synthetic-1 dataset with width $1.5n$.
The matrix becomes more ill-conditioned during training.}
\label{fig:weight_training}
\end{figure}

In \cref{sec:cond_gd}, we observe that ReLU can lead to ill-conditioned
systems.  However, the hidden layer weights were fixed, and one might suspect
that jointly training all of the weights\emthin along with a more sophisticated
gradient method\emthin might lead to better conditioning.
To test this, we jointly learned all weights on the Synthetic-1 dataset
with width $m = 1.5n$, using a stochastic gradient method with
batch size equal to 10,
momentum equal to 0.9,
weight decay equal to 1e-5,
32-bit floats instead of 64-bit floats,
and a learning rate decay of 0.99 after each epoch.
We found that the condition number of $\gamma(W^TX)$ actually increases
during training (see \cref{fig:weight_training}).

\subsection{Landweber iterations in the under-determined case}
\label{sec:Landweber_derivation}

Consider the under-determined linear least-squares problem
\begin{equation}
\min_{\alpha} \frac{1}{2} \| Z \alpha - y \|_2^2,\label{eq:under_lls}
\end{equation}
where $Z$ is an $n \times m$ matrix, $n < m$, with full row rank.
We consider computing the minimizer of \cref{eq:under_lls} using gradient descent with a fixed step size $\eta$
from the starting point $\alpha^{(0)} = 0$. The gradient descent iterates are
\begin{equation}
\label{eqn:GDtoLandweber}
\alpha^{(k+1)} = \alpha^{(k)} - \eta (Z^TZ\alpha^{(k)} - Z^Ty) = \sum_{j=0}^{k} (I - \eta Z^TZ)^j \cdot \eta Z^Ty.
\end{equation}
This is known as the Landweber iteration~\cite[Chapter 6]{hansen2010discrete}.
Landweber iteration is typically analyzed for overdetermined least-squares problems. 
Below, \Cref{lem:Landweber} works out an expression for the iterates in the underdetermined setting.  

\begin{lemma}
\label{lem:Landweber}
Consider $Z\in \R^{n\times m}$ and $y\in\R^m$ with $m\geq n.$ Assume $Z$ has full row-rank, $0 < \eta < \|Z\|_2^2,$ and let $Z=U\Sigma V^T$ be the reduced SVD of $Z.$ The sequence of gradient descent iterates $\{\alpha^{(k)}\}_{k=1}^\infty$ for
\begin{equation*}
\min_{\alpha} \frac{1}{2} \| Z \alpha - y \|_2^2
\end{equation*}
initialized with $\alpha^{(0)} = 0$ satisfy $\alpha^{(k)} = Vf_k(\Sigma)\Sigma^{-1}U^Ty$, where $f_k(z) = 1 - \left(1 - \eta z^2\right)^k$. 
\end{lemma}
\begin{proof}
Similar to~\cref{eqn:GDtoLandweber} we write the gradient descent iterates for $k \geq 1$ as 
\begin{equation*}
\begin{aligned}
\alpha^{(k)} &= \alpha^{(k-1)} - \eta (Z^TZ\alpha^{(k-1)} - Z^Ty)
= \sum_{j=0}^{k-1} (I - \eta Z^TZ)^j \cdot \eta Z^Ty.
\end{aligned}
\end{equation*}
Using the SVD of $Z$, we have that for $k\geq 0$
\begin{equation*}
\begin{aligned}
\alpha^{(k+1)} &=  V \left[ \sum_{j=0}^{k-1} (I-\eta \Sigma^2)^j \right] \eta \Sigma U^T y \\
&=  V \left[ \sum_{j=0}^{k-1} (I-\eta \Sigma^2)^j \eta \Sigma^2\right] \Sigma^{-1} U^T y\\
& = Vf_{k}(\Sigma)\Sigma^{-1} U^T y,
\end{aligned}
\end{equation*}
where for $1-\eta z^2 \neq 1$ 
\begin{equation*}
\begin{aligned}
f_{k}(z) &= \sum_{j=0}^{k-1} (1-\eta z^2)^j \eta z^2 
= \eta z^2 \left(\frac{1 - (1 - \eta z^2)^{k}}{\eta z^2}\right)
= 1 - (1 - \eta z^2)^{k}.
\end{aligned}
\end{equation*}
By assumption $1-\eta z^2 \neq 1$ for $z = \sigma_j,$ which concludes the proof.
\end{proof}

\end{document}